\documentclass[11pt]{amsart}

\usepackage{amsmath}
\usepackage{amssymb}
\usepackage{tikz-cd}

\newtheorem{theorem}{Theorem}[section]
\newtheorem{claim}[theorem]{Claim}

\newtheorem{lemma}[theorem]{Lemma}

\newtheorem{proposition}[theorem]{Proposition}

\theoremstyle{definition}
\newtheorem{definition}[theorem]{Definition}

\newtheorem{question}[theorem]{Question}

\theoremstyle{remark}
\newtheorem{remark}[theorem]{Remark}

\def\smallbox#1{\leavevmode\thinspace\hbox{\vrule\vtop{\vbox
   {\hrule\kern1pt\hbox{\vphantom{\tt/}\thinspace{\tt#1}\thinspace}}
   \kern1pt\hrule}\vrule}\thinspace}

\theoremstyle{definition}

\newcommand{\crit}{\mathop{\mathrm{crit}}}

\newcommand{\Ult}{\operatorname{Ult}}
\newcommand{\Col}{\operatorname{Col}}

\newcommand{\fr}{{}^\frown}

\newcommand{\name}{\dot}
\newcommand{\can}{\check}

\newcommand{\la}{\langle}
\newcommand{\ra}{\rangle}

\newcommand{\uhr}{\restriction}

\newcommand{\ol}{\ol}
\newcommand{\po}{\mathbb{P}}
\newcommand{\qo}{\mathbb{Q}}

\newcommand{\A}{\mathcal{A}}

\renewcommand{\ol}{\overline}

\renewcommand{\succ}{\suc}
\newcommand{\otp}{\mathop{\mathrm{ot}}}
\newcommand{\Add}{\mathop{\mathrm{Add}}}

\DeclareMathOperator{\dom}{dom}
\DeclareMathOperator{\range}{range}
\DeclareMathOperator{\cf}{cf}
\DeclareMathOperator{\mc}{mc}

\DeclareMathOperator{\suc}{succ}
\DeclareMathOperator{\GCH}{GCH}
\DeclareMathOperator{\SCH}{SCH}

\DeclareMathOperator{\cof}{Cof}

\def\l{{\langle}}
\def\r{{\rangle}}

\title{Stationary Reflection and the Failure of the SCH}
\date{}
\thanks{The first author was partially supported by the Israel Science Foundation Grant 1832/19.}
\thanks{The second author was partially supported by FWF Lise Meitner grant 2650-N35 and the Israel Science Foundation Grant 1967/21.}
\thanks{The third author was partially supported by NSF grant DMS-1700425 and an NSERC Discovery grant.}

\author{Omer Ben-Neria}
\address{Einstein Institute of Mathematics,
 The Hebrew University of Jerusalem,
 Jerusalem 91904, Israel}
\email{omer.bn@mail.huji.ac.il}

\author{Yair Hayut}
\address{Einstein Institute of Mathematics,
 The Hebrew University of Jerusalem,
 Jerusalem 91904, Israel}

\email{yair.hayut@mail.huji.ac.il}

\author{Spencer Unger}
\address{Department of Mathematics, 
University of Toronto, Toronto, ON, Canada}
\email{unger.the.aronszajn.trees@gmail.com}

\begin{document}
\begin{abstract}
In this paper we prove that from large cardinals it is consistent that there is
a singular strong limit cardinal $\nu$ such that the singular cardinal
hypothesis fails at $\nu$ and every collection of fewer than $\cf(\nu)$
stationary subsets of $\nu^+$ reflects simultaneously.  For $\cf(\nu) > \omega$,
this situation was not previously known to be consistent.  Using different
methods, we reduce the upper bound on the consistency strength of this situation
for $\cf(\nu) = \omega$ to below a single partially supercompact cardinal.  The
previous upper bound of infinitely many supercompact cardinals was due to
Sharon.  \end{abstract} \maketitle

\noindent

\section{Introduction}

We study stationary reflection at successors of singular cardinals and its
connection with cardinal arithmetic.  We start by recalling some basic
definitions.  For an ordinal $\delta$ and a set $S \subseteq \delta$, we say that
$S$ is \emph{stationary} if it meets every closed and unbounded subset of
$\delta$.  If $\{S_i \mid i \in I\}$ is a collection of stationary subsets of a
regular cardinal $\kappa$, then we say that $\{S_i \mid i \in I\}$ reflects
simultaneously if there is an ordinal $\delta$ such that $S_i \cap \delta$ is
stationary for all $i \in I$.

The consistency of stationary reflection at the successor of singular cardinal
is already complex in the context of the generalized continuum hypothesis (GCH).
A theorem of Magidor \cite{magidor} shows that it is consistent relative to
the existence of infinitely many supercompact cardinals that every finite
collection of stationary subsets of $\aleph_{\omega+1}$ reflects.  Recently, the
second and third author \cite{HU} were able show the same result from an
assumption below the existence of a cardinal $\kappa$ which is
$\kappa^+$-supercompact.  Both of these models satisfy GCH.  Combining
stationary reflection at the successor of a singular cardinal with the failure
of $\SCH$ presents additional difficulties.

For a singular cardinal $\nu$, the singular cardinal hypothesis ($\SCH$) at $\nu$
is the assertion that if $\nu$ is strong limit, then $2^\nu = \nu^+$.  The
failure of the singular cardinal hypothesis at a singular cardinal $\nu$ is
known to imply the existence of many nonreflecting objects.  For instance,
Foreman and Todorcevic \cite{FT} have shown that the failure of the singular
cardinal hypothesis at $\nu$ implies that there are two stationary subsets of
$[\nu^+]^\omega$ which do not reflect simultaneously.  This was improved
by Shelah \cite{shelah} to obtain a single stationary subset of $[\nu^+]^\omega$
which does not reflect.  Reflection for stationary subsets of $[\nu^+]^\omega$
has a different character than reflection for stationary subsets of ordinals.

In \cite[Corollary 5.4]{SigmaPrikryI}, Poveda, Rinot and Sinapova showed that the failure of $\SCH$ at cardinal $\nu$ implies the failure of simultaneous reflection for collections of $\cf \nu$ many stationary subsets of $\nu^+$. 
In his PhD thesis from 2005, Sharon \cite{sharon} proved that relative to the existence of
infinitely many supercompact cardinals it is consistent that there is a singular
cardinal $\nu$ of cofinality $\omega$ such that $\SCH$ fails at $\nu$ and every
stationary subset of $\nu^+$ reflects.  Sharon's method is a tour de force construction, which builds on Gitik's long extenders forcing \cite{Gitik-HB} for $\omega$-sequence of hypermeasurable cardinals. As such, the construction does not extend to singular cardinals of uncountable cofinalities, and the question of whether the failure of $\SCH$ at singular $\kappa$ of uncountable cofinality together with stationary reflection at $\kappa^+$ is at all consistent.  
 
This paper follows a study by the authors on stationary reflection at successors of singular cardinals at which $\SCH$ fails. 
This study was prompted by two other recent studies.  First, the work of second
and third authors in \cite{HU} on stationary reflection in Prikry forcing
extensions from subcompactness assumptions. The arguments of \cite{HU} show how
to examine the stationary reflection in the extension by Prikry type forcing by
studying suitable iterated ultrapowers of $V$. This approach and method has been
highly effective in our situation and we follow it here. The second study is
Gitik's recent work \cite{Gitik-new} for blowing up the power of a singular
cardinal using a Mitchell order increasing sequence of overlapping extenders.
The new forcing machinery of \cite{Gitik-new} gives new models combining the
failure of $\SCH$ with reflection properties at successors of singulars.
Moreover the arguments are uniform in the choice of cofinality.  For example,
Gitik has shown that in a related model the tree property holds at $\kappa^+$
\cite{gitiktree} and that for all $\delta < \kappa$ there is a stationary subset
of $\kappa^+$ of ordinals of cofinality greater than $\delta$ which is not a
member of $I[\kappa^+]$ \cite{gitikap}.

Our first theorem provides a model for stationary reflection the successor of a
singular cardinal $\nu$ where $\SCH$ fails and the cofinality of $\nu$ can be some
arbitrary cardinal chosen in advance.

\begin{theorem}\label{mainthm} Let $\eta < \lambda$ be regular cardinals.
Suppose that there is an increasing sequence of cardinals $\langle \kappa_\alpha
\mid \alpha < \eta \rangle$ with
\begin{enumerate}
\item $\eta< \kappa_0$,
\item for each $\alpha < \eta$, $\kappa_\alpha$
carries a $(\kappa_\alpha,\lambda)$-extender $E_\alpha$ and there is a
supercompact cardinal between $\sup_{\beta<\alpha}\kappa_\beta$ and
$\kappa_\alpha$, and 
\item the sequence $\langle E_{\alpha} \mid \alpha < \eta \rangle$ is Mitchell
order increasing and coherent.
\end{enumerate}
Then, there is a cardinal and cofinality preserving extension in which, setting
$\kappa = \sup_{\alpha < \eta}\kappa_\alpha$, $2^{\kappa} = \lambda$ and every
collection of fewer than $\eta$ stationary subsets of $\kappa^+$ reflects. \footnote{After proving this theorem, Gitik also gave an alternative proof for the same result, see \cite{GitikReflNotSCH}}
\end{theorem}

We have the following improvement of Gitik's result about $I[\kappa^+]$.

\begin{theorem} \label{mainthm2} Let $\langle \kappa_\alpha \mid \alpha < \eta
\rangle$ be as in Theorem \ref{mainthm} with the exception that we only assume
there is a single supercompact cardinal $\theta<\kappa_0$.  There is a cardinal
and cofinality preserving extension in which setting $\kappa =
\sup_{\alpha<\eta}\kappa_\alpha$, we have that $\kappa$ is strong limit,
$2^\kappa =\lambda$ and there is a scale of length $\kappa^+$ such that the set
of nongood points of cofinality less than $\theta$ is stationary. \end{theorem}

After a suitable preparation, the extension is obtained using the same forcing
as in Theorem \ref{mainthm}.  Standard arguments show that in this model, the
set of nongood points is stationary in cofinalities that are arbitrarily high
below $\theta$.  Further, the same argument shows that we can take $\theta$
between $\sup_{\alpha<\beta}\kappa_\alpha$ and $\kappa_\beta$ and reach the same
conclusion.

 These two results continue a long line of research about the
interaction between the failure of the singular cardinal hypothesis and weak
square-like principles, \cite{gitiksharon, neeman, sinapova3,sinapova1,
sinapova2,sinapovaunger}.

We do not know whether these results can be also obtained at small singular cardinals of uncountable cofinalites. For concreteness we suggest the following question. 
\begin{question}
 Is it consistent that $\SCH$ fails $\aleph_{\omega_1}$ and every stationary subset of $\aleph_{\omega_1+1}$ reflects?
\end{question}

We also give another model for stationary reflection at $\nu^+$ where $\nu$ is a
singular cardinal of cofinality $\omega$ where the singular cardinal hypothesis
fails. This construction replaces the supercompactness assumption in Sharon's result and Theorem \ref{mainthm} with the weaker one of subcompactness together with hypermeasurability.

\begin{theorem}\label{mainthm3} Suppose that $\kappa$ is
$\kappa^+$-$\Pi_1^1$-subcompact and carries a $(\kappa,\kappa^{++})$-extender.
There is a forcing extension in which $\kappa$ is singular strong limit of
cofinality $\omega$, $2^\kappa=\kappa^{++}$ and every finite collection of
stationary subsets of $\kappa^+$ reflects simultaneously. \end{theorem}

In Theorem \ref{thm:down-to-aleph-omega}, we will also show how to collapse cardinals and make $\kappa$ in Theorem \ref{mainthm3} to be $\aleph_{\omega}$.\footnote{This theorem did not appear in the original version of this paper.} 

The construction and proof follows the lines of the second and third authors`
paper \cite{HU}, and the work of Merimovich \cite{Merimovich} on generating
generics for extender based forcing over iterated ultrapowers.  The large
cardinal assumption in the theorem is the natural combination of the assumption
from \cite{HU} and an assumption sufficient to get the failure of $\SCH$ by
extender based forcing.  Unfortunately, we are unable to adapt the argument from
the previous theorem to a singular cardinal of uncountable cofinality.  We ask

\begin{question}
 It is possible to obtain the result of Theorem \ref{mainthm} without any supercompactness assumptions? Namely, is it possible to start with a single partially supercompact cardinal $\kappa$ and force $\kappa$ to be a strong limit singular of uncountable cofinality such that every stationary subset of $\kappa^+$ reflects? Is it possible to obtain that together with the failure of $\SCH$ at $\kappa$?
\end{question}


The paper is organized as follows.  In Section \ref{gitik}, we define Gitik's
forcing from \cite{Gitik-new}, which will be used in our main theorem.  In
Section \ref{unctble}, we prove that in mild generic extensions of $V$ we can
find a generic for Gitik's forcing over a suitable iterated ultrapower.  In
Section \ref{reflection}, we give the proof of the main theorem by arguing that
stationary reflection holds in the generic extension of the iterated ultrapower
constructed in the previous section. In Section \ref{badscale}, we prove Theorem
\ref{mainthm2}.  In Section \ref{ctble}, we give the proof of Theorem
\ref{mainthm3} and its variation for $\aleph_\omega$.

\section{Gitik's forcing} \label{gitik}
In this section we give a presentation of Gitik's forcing \cite{Gitik-new} for
blowing up the power of a singular cardinal with a Mitchell order increasing
sequence of extenders.  

Let us start with the following definitions:
\begin{definition}
Let $E_0, E_1$ be $(\kappa_0, \lambda_0)$ and $(\kappa_1, \lambda_1)$-extenders respectively. We say that $E_0$ is less than $E_1$ in the Mitchell order, or $E_0 \trianglelefteq E_1$, if $E_0 \in \Ult(V, E_1)$. We say that $E_0$ coheres with $E_1$ if $j_{E_1}(E_0) \restriction \lambda_0 = E_0$.   
\end{definition}

The existence of a long Mitchell increasing and pairwise coherent sequence of extenders $E_i$, where $E_i$ is a $(\kappa_i, \lambda)$-extender, follows from the existence of superstrong cardinal or even weaker large cardinal axioms.

Before we begin with the definition of the forcing, let us show a few basic
facts about extenders and Mitchell order.  We recall the notion of width from
\cite{CummingsHandbook}.

\begin{definition} Let $k\colon M \to N$ be an elementary embedding between transitive
models of set theory and let $\mu$ be an ordinal.  We say the embedding $k$ has
width $\leq\mu$ if every element of $N$ is of the form $k(f)(a)$ for some $f \in
M$ and $a \in N$ such that $M \models \vert \dom(f)\vert \leq \mu$. \end{definition}

\begin{lemma}\label{lemma: cofinal generators}
Let $E$ be a $(\kappa,\lambda)$-extender and let $\leq_E$ be the Rudin-Keisler order of the extender $E$.
Let $k\colon V \to M$ be an elementary embedding with width $\leq \kappa$. Then 
the set $k `` \dom E$ is $\leq_{k(E)}$-cofinal in $\dom k(E)$. 
\end{lemma}
\begin{proof}
By \cite{Gitik-HB}, the Rudin-Keisler order $\leq_E$ is $\kappa^{+}$-directed.

Let $a \in \dom k(E)$. Then, by the definition of width, there is a function $f\colon \kappa \to \dom E$ and some generator $b$ such that $k(f)(b) = a$. Let $a'$ be a $\leq_E$ upper bound of the image of $f$. Then, clearly, $k(a')$ is $\leq_{k(E)}$ above $a$. 
\end{proof}
\begin{lemma}\label{lemma: commutative} 
Let $E_0$ be a $(\kappa_0, \lambda)$-extender and let $E_1$ be a $(\kappa_1,\lambda)$-extender, where $\kappa_0 \leq \kappa_1$. Let us assume that $E_0 \trianglelefteq E_1$. Then the following diagram commutes:

\[
\begin{tikzcd}
V\arrow{r}{E_0}\arrow{d}{E_1} & M_0\arrow{d}{j_{E_0}(E_1)} \\
M_1\arrow{r}{E_1} & N
\end{tikzcd}
\]
where each arrow represents the ultrapower map using the indicated extender.
\end{lemma}
\begin{proof}
First, since $E_0\in M_1$, all maps are internally defined and in particular, all models are well founded. 

Moreover, 
\[\left(j_{E_0}\right)^V \restriction M_1 = \left(j_{E_0}\right)^{M_1}.\]
In order to verify this, it is sufficient to show that those maps are the same on sets of ordinals. First, let us verify that they agree of the ordinals. Indeed, let us consider the class:
\[\{(f,a) \mid f \colon \kappa_0^{<\omega} \to \mathrm{Ord}\},\]
ordered using the extender $E_0$: \[(f_0,a_0) \leq (f_1, a_1) \iff j_{E_0}^V(f_0)(a_0) \leq j_{E_0}^V(f_1)(a_1).\]
Using the combinatorial definition of the extender ultrapower, we conclude that $M_1$ can compute this order correctly, and in particular, it computes correctly the ordertype of the elements below any constant function, or indeed below any element of the form $[f,a]$ that represent an ordinal. Indeed, as the ultrapower by $E_1$ is closed under $\kappa_0$-sequences, it contains the above class. Since $E_0 \trianglelefteq E_1$, $M_1$ contains the relevant measures and Rudin-Keisler projections. 

For an arbitrary set of ordinals, $y$, one can easily compute whether $[f,a] \in [c_y, \{\kappa\}]$ using the combinatorial information of the extender.

 Let us consider an element of $\Ult(M_1, E_0)$. By the definition, it has the form: $x = j_{E_0}(g)(a_0)$ where $g\colon \kappa_0^{<\omega} \to M_1$. Going backwards, we can find a function in $V$, $f$, and a generator $a_1$ such that for every $z$, $j_{E_1}(f)(z, a_1) = g(y)$ and therefore:
\[x = j_{E_0}\left( j_{E_1}(f)( - , a_1) \right)(a_0) = j_{j_{E_0}(E_1)} \big( j_{E_0} (f)\big)\left(a_0, j_{E_0}(a_1)\right).\]
This element in obviously in $\Ult(M_0, j_{E_0}(E_1))$. 

On the other hand, if $y \in \Ult(M_0, j_{E_0}(E_1))$ then there is some generator $a_1' \in \dom j_{E_0}(E_1)$ and a function $g$ such that $y = j_{j_{E_0}(E_1)}(g)(a_1')$. By Lemma \ref{lemma: cofinal generators}, we may assume that $a_1' = j_{E_0}(a_1)$. Let $f$ be a function in $V$ and $a_0$ be some generator such that $g = j_{E_0}(f)(a_0)$, then we have:
\[
\begin{matrix}
y & = & j_{j_{E_0}(E_1)}\left(j_{E_0}(f)(a_0, - )\right)(j_{E_0}(a_1)) \\
 & = & j_{j_{E_0}(E_1)}\left(j_{E_0}\left(f(-, a_1)\right)(a_0)\right) \\
 & = & j_{E_0}\big(j_{E_1}(f)( - , a_1)\big)(a_0),
\end{matrix}
\]
as wanted.
\end{proof}
Let $\langle \kappa_\alpha \mid \alpha < \eta \rangle$ 
be a sequence of cardinals as in Theorem \ref{mainthm}.  Following 
Merimovich (see for example \cite{Merimovich}), we can assume that the extenders
$E_\alpha$ are of the form $\langle E_\alpha(d) \mid d \in [\lambda]^{\kappa_\alpha}
\rangle$ where for $X \subseteq \{f\mid \dom f \subseteq d,\, \range f \subseteq \kappa_\alpha,\, |f| < \kappa_\alpha\}$, $X \in E_\alpha(d)$ if and only
if $\{ (j_{E_\alpha}(\xi),\xi) \mid \xi \in d \} \in j_{E_\alpha}(X)$.  

For $d\in [\lambda]^{\kappa_\alpha}$ with $\kappa_\alpha \in d$, it is
easy to see that the measure $E_\alpha(d)$ concentrates on a set of order
preserving functions $\nu$ with $\kappa \in \dom(\nu)$.  So we assume that every
measure one set mentioned below is of this form.  We also fix functions $\la
\ell_\alpha \mid \alpha < \eta\ra$ so that for every $\alpha <  \eta$,
$j_{E_\alpha}(\ell_\alpha)(\kappa_\alpha) = \lambda$.  The existence of such
functions can always be arranged by a simple preliminary forcing.

Using the set $d \in [\lambda]^{\kappa_\alpha}$ to index the extenders $E_\alpha$ has
the advantage that the projection maps from $E_\alpha(d')$ and $E_\alpha(d)$ for
$d \subseteq d'$ can be made very explicit.  The measure $E_\alpha(d')$
concentrates on partial functions from $d'$ to $\kappa_\alpha$ with domain smaller than
$\kappa_\alpha$.  Thus, the map $\nu \to \nu \restriction d$ is a projection from
$E_\alpha(d')$ to $E_\alpha(d)$.

For every two cardinals $\mu < \lambda$, let $\A(\lambda,\mu)$ be the
poset consisting of partial functions $f \colon \lambda \to \mu$ with $|f| \leq
\mu$ and $\mu \in \dom(f)$. Therefore $\A(\lambda,\mu)$ is isomorphic
to Cohen forcing for adding $\lambda$ many subsets of $\mu^+$.

We let $\mathbb{P}$ be Gitik's forcing from \cite{Gitik-new} defined from the
sequence of extenders $\langle E_\alpha \mid \alpha < \eta \rangle$.  We give a
compact description of the forcing.  A condition $p \in \mathbb{P}$ is a
sequence $\langle p_\alpha \mid \alpha < \eta \rangle$ such that there is a
finite set $s^p \subseteq \eta$ such that for each $\alpha < \eta$, $p_\alpha =
(f_\alpha,\lambda_\alpha)$ if $\alpha \in s^p$, and $p_\alpha =
(f_\alpha,A_\alpha)$ otherwise, and the following conditions hold.
\begin{enumerate}
\item $f_\alpha \in \A(\lambda_{\alpha^*},\kappa_\alpha)$ where $\alpha^*$ is
the next element of $s^p$ above $\alpha$ if it exists and $f_\alpha \in
\A(\lambda,\kappa_\alpha)$ otherwise.

\item For all $\alpha \in s^p$, $\lambda_\alpha$ is a cardinal and
$\sup_{\beta< \alpha}\kappa_\beta < \lambda_\alpha < \kappa_\alpha$.

\item For all $\alpha \in \eta \setminus s^p$, if $\alpha > \max(s^p)$, then
$A_\alpha \in E_\alpha(\dom(f_\alpha))$, otherwise if $\alpha^*$ is the least
element of $s^p$ above $\alpha$, then $A_\alpha \in E_\alpha(\dom(f_\alpha))$.

\item For $\alpha \notin s^p$, $f_\alpha^p(\kappa_\alpha)
=0$. \footnote{This gives a clean way to distinguish between Cohen functions associated to
members of $s^p$ and those which are not.}

\item The sequence $\langle \dom(f_\alpha) \mid \alpha < \eta \rangle$ is
increasing.
\end{enumerate}

We adopt the convention of adding a superscript $f_\alpha^p$, $A_\alpha^p$, etc.
to indicate that each component belongs to $p$.  When the value of
$\lambda_{\alpha^{*}}$ might behave non-trivially, we will add it as a third
coordinate to the pairs $p_\alpha$, where $\alpha \notin s^p$.  We call
$\eta\setminus s^p$ the \emph{pure part} of $p$ and $s^p$ the \emph{non-pure
part} of $p$.

We briefly sketch the notion of extension. A condition $p$ is a direct extension of $q$ if
$s^p = s^q$ and for all $\alpha$, $f_\alpha^p$
is stronger than $f_\alpha^q$ and $A_\alpha^p$ projects to a subset of $A_\alpha^q$ using the natural Rudin-Keisler projection from $E_\alpha(\dom f_\alpha^p)$ to $E_\alpha(\dom f_\alpha^q)$.

Let us describe now the one point extension.  Suppose that $\nu \in A_\beta^p$.
We let $q = p \frown \nu$ be the condition with $s^q = s^p \cup \{\beta\}$ and
the following.

\begin{enumerate}
\item $f_\beta^q = f_\beta^p \fr \nu$ is the overwriting of $f_\beta^p$ by
$\nu$, that is 
\[ (f_\beta^p\fr \nu)(\tau) = 
\begin{cases}
 \nu(\tau) \mbox{ if } \tau \in \dom(\nu)\\
 f^p_\beta(\tau) \mbox{ otherwise. }
\end{cases}
\]
\item $\lambda_\beta^q = \ell_\beta(\nu(\kappa_\beta))$.
\item For $\alpha \in [\max(s^p) \cap \beta,\beta)$, $f_\alpha^q = 
 f_\alpha^p \circ \nu^{-1}$ and $A_\alpha^q = \{ \xi \circ \nu^{-1}  \mid \xi \in
A_\alpha^p \}$ if applicable.
\end{enumerate}

The following analysis of dense open sets of $\po$ was
established in \cite{Gitik-new}.

\begin{lemma}\label{Lem-meetdense} For every condition $p \in \po$ and dense
open set $D \subseteq \po$, there are $p^* \geq^* p$ and a finite subset $\{
\alpha_0, \dots, \alpha_{m-1}\}$ of the pure part of $p$, such that for
every sequence $\vec{\nu} = \la \nu_{\alpha_0},\dots, \nu_{\alpha_{m-1}}\ra \in
\prod_{i< m} A^{p^*}_{\alpha_i}$, $p^* \fr \vec{\nu} \in D$.  \end{lemma}

For limit $\delta \leq \eta$, we define $\bar{\kappa}_\delta =
\sup_{\alpha<\delta}\kappa_\alpha$.  Gitik used the previous lemma to prove:

\begin{theorem} In the generic extension by $\po$, cardinals and cofinalities
are preserved and for every limit $\delta \leq \eta$, the singular cardinal
hypothesis fails at $\bar{\kappa}_\delta$. \end{theorem}

For use later, we define $\vec{\A}$ to be the full-support product
$\prod_{\alpha < \eta} \A(\lambda,\kappa_\alpha)$, and similarly, for each
$\beta < \eta$, $\vec{\A}_{\geq\beta} = \prod_{\beta \leq \alpha < \eta}
\A(\lambda,\kappa_\alpha)$. We aim to show that if $\vec{H} = \la H(\alpha) \mid
\alpha < \eta\ra$ is $\vec{\A}$ generic over $V$, then there is a generic for
the image of $\po$ in a suitable iterated ultrapower. 

\section{The iterated ultrapower $M_{\eta}$ and the generic filter
$G^*$}\label{unctble}

We consider the following iterated ultrapower \[\l M_\alpha,j_{\alpha,\beta} \mid \alpha \leq \beta \leq \eta \r\] by the extenders in $\vec{E}$. The iteration is defined by induction of $\alpha$. 
Let $M_0 = V$. For every $\alpha < \eta$, given that $j_{0,\alpha} \colon  M_0 \to M_\alpha$ has been defined, we take $E^\alpha_\alpha = j_{0,\alpha}(E_\alpha)$, 
and set $j_{\alpha,\alpha+1} \colon  M_\alpha \to M_{\alpha+1} \cong \Ult(M_\alpha,E_{\alpha}^\alpha)$. At limit stage $\delta \leq \eta$ we take $M_\delta$ to be the 
direct limit of the system $\l M_\alpha,j_{\alpha,\beta} \colon  \alpha \leq \beta  <\delta\r$, and $j_{\alpha, \delta}$ to be the limit maps.

For every $\beta \leq \eta$, we shall abbreviate and write $j_\beta$ for
$j_{0,\beta}$.  For a given $\beta \leq \eta$ we shall denote
$j_{\beta}(\kappa_\alpha), j_{\beta}(E_\alpha)$, $j_{\beta}(\lambda)$ by
$\kappa_\alpha^\beta$, $E_\alpha^\beta$, and $\lambda^\beta$ respectively.
Similarly, we will denote $j_{\beta}(x) = x^\beta$ for
objects $x \in V$ whose images along the iteration will be significant for our
construction. 

\begin{lemma}\label{Lem-itergens}
 Let $\beta \leq \eta$.   For every $x \in M_{\beta}$ there are $\beta_0< \beta_1 < \dots < \beta_{l-1}$ below $\beta$, 
 ordinals $\tau_0,\dots,\tau_{l-1}$ below $\lambda$, and $f  \colon \prod_{i < l} \kappa_{\beta_i} \to V$, such that 
 \[ x = j_{\beta}(f)\left(j_{\beta_0}(\tau_0), j_{\beta_1}(\tau_1), \dots, j_{\beta_{l-1}}(\tau_{l-1})\right).\]
\end{lemma}
\begin{proof}
  It is immediate from the definition of the iteration that for every $x \in M_\beta$ there are $\beta_0 < \dots < \beta_{l-1}$ below $\beta$, 
 finite subsets of ordinals $a_0,\dots, a_{l-1}$, with each $a_i \in [\lambda^{\beta_i}]^{<\omega}$ a generator of $E_{\beta_i}^{\beta_i}$, 
 and a function $g \colon  \prod_{i < l} \kappa_{\beta_i}^{<\omega} \to V$ so that 
 \[x = j_{0,\beta}(g)(j_{\beta_0+1,\beta}(a_0), \dots, j_{\beta_{l-1}+1,\beta}(a_{l-1})).\]

 Note that for each $i < l$, $j_{\beta_i+1,\beta}(a_i) = a_i$, because \[\crit(j_{\beta_i+1, \beta}) = \kappa_{\beta_i+1}^{\beta_{i}+1} > \lambda^{\beta_i}.\]
 
 It follows that
 \[x = j_{\beta}(g)(a_0, a_1, \dots, a_{l-1}).\]

 We may further assume that each $a_i < \lambda^{\beta_i}$ is an ordinal. 
 Finally, we claim that we may replace each $a_i$ with an ordinal of the form
$j_{\beta_i}(\tau_i)$, for some $\tau_i < \lambda$. 
 By Lemma \ref{lemma: cofinal generators}, in $M_{\beta_i}$, the extender $E_{\beta_i}^{\beta_i} = j_{\beta_i}(E_{\beta_i})$ is generated by 
 a subset of its measures, which are of the form $j_{\beta_i}(E_{\beta_i}(\tau_i)) = E_{\beta_i}^{\beta_i}(j_{\beta_i}(\tau_i))$. 
 Namely, for every $a_i \in \lambda^{\beta_i}$ there exists some $\tau_i \in \lambda$ such that 
 $a_i \leq^{RK}_{E_{\beta_i}^{\beta_i}} j_{\beta_i}(\tau_i)$. 

 In particular \[a_i = j_{\beta_i}(h)(b_1,\dots,b_k) \leq^{RK}_{E_{\beta_i}^{\beta_i}}  j_{\beta_i}(\tau_i).\]  The claim follows.
\end{proof}

Having defined the iterated ultrapower $M_{\eta}$ we proceed to introduce
relevant conditions in $j_{\eta}(\po)$.  Let $p = \l p_\alpha \mid \alpha <
\eta\r$ be a pure condition in $\po$, that is, a condition with $s^p
= \emptyset$.  For every $\alpha < \eta$, we set $p_\alpha = \la
f_\alpha,A_\alpha,\lambda\ra$, where $f_\alpha \in A(\lambda,\kappa_\alpha)$ and
$A_\alpha \in E_\alpha(d_\alpha)$ where $d_\alpha = \dom(f_\alpha) \in
[\lambda]^{\leq \kappa_\alpha}$.

Since $j_{\alpha,\alpha+1} =
j^{M_\alpha}_{E_\alpha^\alpha}$ it follows that the function
\[ \nu^{p^\alpha}_{\alpha,\alpha+1} := (j_{\alpha,\alpha+1}\uhr d_\alpha^\alpha)^{-1} \]
belongs to $j_{\alpha,\alpha+1}(A^\alpha_\alpha) = A^{\alpha+1}_\alpha$, which is the measure one set of the $\alpha$-th component of $j_{\alpha+1}(p)$.
By applying $j_{\alpha+1,\eta}$ to $j_{\alpha+1}(p)$, we conclude that the function 
\[\nu^{p^\alpha}_{\alpha,\eta} := j_{\alpha+1,\eta}(\nu^{p^\alpha}_{\alpha,\alpha+1}) = (j_{\alpha,\eta}\uhr d_\alpha^\alpha)^{-1}\]
belongs to $j_{\eta}(A_\alpha)$, and thus $j_{\eta}(p)$ has an one point
extension at the $\alpha$-th coordinate of the form 
\[j_{\eta}(p) \fr \la \nu^{p^\alpha}_{\alpha,\eta}\r.\]
It is clear that for every finite sequence $\alpha_0 < \dots < \alpha_{m-1}$ of
ordinals below $\eta$, and pure condition $p \in \po$, we have
$j_{\eta}(p) \fr \l \nu^{p^{\alpha_0}}_{\alpha_0,\eta}, \nu^{p^{\alpha_1}}_{\alpha_1,\eta}, \dots , \nu^{p^{\alpha_{m-1}}}_{\alpha_{m-1},\eta}\r$
is an extension of $j_{\eta}(p)$ in $j_{\eta}(\po)$. 
It is also clear that if $p,q \in \po$ are two pure conditions so that for all $\alpha < \eta$, $f_\alpha^p,f_\alpha^q$ are compatible, 
then for every two sequences $\alpha_0 < \dots < \alpha_{m-1}$ and $\beta_0 < \dots < \beta_{l-1}$ of ordinals below $\eta$, the conditions
\[j_{\eta}(p) \fr \l \nu^{p^{\alpha_0}}_{\alpha_0,\eta}, \nu^{p^{\alpha_1}}_{\alpha_1,\eta}, \dots , \nu^{p^{\alpha_{m-1}}}_{\alpha_{m-1},\eta}\r\]
and
\[j_{\eta}(q) \fr \l \nu^{q^{\beta_0}}_{\beta_0,\eta}, \nu^{q^{\beta_1}}_{\beta_1,\eta}, \dots , \nu^{q^{\beta_{l-1}}}_{\beta_{l-1},\eta}\r\] are compatible 
in $j_{\eta}(\po)$. We remark that in the assumption that the sequence of extenders
$\l E_\alpha \mid \alpha < \eta \r$ 
is Mitchell order increasing is used in order to be able to permute the order
in which the extensions are done, using a sequence of applications of Lemma \ref{lemma: commutative}.

Fix $\vec{H} = \l H(\alpha) \mid \alpha < \eta\r$ a $V$-generic for $\vec{\A} =
\prod_{\alpha < \eta} \A(\lambda,\kappa_\alpha)$. 

\begin{definition}\label{definition:G*-for-Gitik-forcing}
Define $G^* \subset j_{\eta}(\po)$ to 
be the filter generated by all conditions 
\[j_{\eta}(p) \fr \l \nu^{p^{\alpha_0}}_{\alpha_0,\eta}, \nu^{p^{\alpha_1}}_{\alpha_1,\eta}, \dots , \nu^{p^{\alpha_{m-1}}}_{\alpha_{m-1},\eta}\r,\] 
where $p \in \po$ satisfy that $f^p_\alpha \in H(\alpha)$ for all $\alpha < \eta$, $m < \omega$, and $\alpha_0 < \dots < \alpha_{m-1}$ are ordinals below $\eta$. 
\end{definition}

Our first goal is to prove that $G^*$ generates a $j_{\eta}(\po)$
generic filter over $M_{\eta}$. 

\begin{proposition}\label{Prop-G*generic}
 $G^*_{\eta}$ is $j_{\eta}(\po)$ generic over $M_{\eta}$. 
\end{proposition}

Before moving to the proof of the proposition, we discuss finite subiterates of $M_{\eta}$.

\subsection{Finite sub-iterations $N^F$}\label{Rem-finitesubiter}  The model
$M_{\eta}$ can be seen as a directed limit of all its finite subiterates, $N^F$,
$F \in [\eta]^{<\omega}$.  Given a finite set $F = \{ \beta_0,\dots,\beta_{l-1}
\} \in [\eta]^{<\omega}$, we define its associated iteration $\la N^F_{i},
i^F_{m,n} \mid m \leq n \leq l\ra$ by $N^F_0 = V$, and $i^F_{n,n+1} \colon  N^F_n \to
N^F_{n+1} \cong \Ult(N^F_{n}, i^F_n(E_{\beta_{n}}))$.  Since for the moment we
handle a single finite set at a time, we will suppress the mention of the finite
set $F$ and refer only to the iteration as $i_{m,n}\colon N_m \to N_n$ for $m \leq n
\leq l$ and as usual we set $i_n = i_{0,n}$ for $n \leq l$.

The proof of Lemma \ref{Lem-itergens} shows that the elements of $N_l$ are of
the form $i_l(f)\left(\tau_0, i_{1}(\tau_1), \dots ,
i_{l-1}(\tau_{l-1})\right)$.  Let $k\colon  N_l \to M_{\eta}$ be the usual factor
map defined by
\[ k(i_l(f)\left(\tau_0, i_{1}(\tau_1), \dots,
i_{l-1}(\tau_{l-1})\right)) = j_{\eta}(f)\left(j_{\beta_0}(\tau_0), j_{\beta_1}(\tau_1), \dots, j_{\beta_{l-1}}(\tau_{l-1})\right)\]

It is routine to verify that $k$ is well defined, elementary and $j_{\eta} = k \circ i_l$.

Moreover, following this explicit description of $k$, it is straightforward to
verify that $k : N_l \to M_{\eta}$ is the resulting iterated ultrapower limit
embedding, associated to the iteration of $N_l$ by the sequence 
\[\l i_{0,n_\alpha}(E_\alpha) \mid \alpha \in \eta \setminus \{\beta_0, \dots
\beta_{l-1}\}\r,\] where $n_\alpha$ is the minimal $n < l$ for which $\beta_n \geq
\alpha$, if exists, and $n_\alpha = l$ otherwise. This again uses the Mitchell order
assumption of the sequence of extenders in order to get the desired commutativity.

Next, we observe that our assignment of generators 
$\nu^{p^\alpha}_{\alpha,\eta}$, $\alpha < \eta$, to $j_{\eta}(p)$ of conditions
$p \in \po$, in Definition \ref{definition:G*-for-Gitik-forcing}, can be defined at the level of the finite subiterates.

Indeed, for a condition $p \in \po$, $n < l$, we temporarily define
\[ \nu^{i_n(p)}_{n,n+1} = (i_{n,n+1}\uhr d)^{-1} = \{ (i_{n,n+1}(\tau), \tau) \mid \tau \in d\},\]
where $d = \dom(f^{i_n(p)})_{\beta_n}$.  We let $\nu^{i_n(p)}_{n,l} =
i_{n+1,l}(\nu^{i_n(p)}_{n,n+1})$ be the natural push forward to $N_l$.  As with
the conditions $j_{\eta}(p)$, we have that \[i_l(p) \fr \la \nu^{i_0(p)}_{0,l},
\dots, \nu^{i_{l-1}(p)}_{l-1,l} \ra\] is a valid extension of $i_l(p)$ in
$i_l(\po)$.

To record this definition (which depends on $F$), we make a few permanent
definitions which explicitly mention $F$ as well as objects from the discussion above: 
\begin{definition}\label{definition:natural-non-pure-extension}
Fix $i^F =
i_l$ and $k^F=k$ and the model $N^F = N_l$. Denote the condition defined
above by $i^F(p) \fr \vec{\nu}^{p,F}$ and refer to it as the \emph{natural
non-pure extension of} $i^F(p)$.
\end{definition}
Using the description of $k^F \colon N^F \to M_{\eta}$, it is straightforward
to apply {\L}o\'{s}'s Theorem and show that $k^F(\vec{\nu}^{p,F}) = \l
\nu^{p^{\beta_n}}_{\beta_n,\eta}\mid n < l \r$. We conclude that 
\[ 
k^F\left( i^F(p) \fr \vec{\nu}^{p,F}  \right) = 
 j_{\eta}(p) \fr \l \nu^{p^{\beta_0}}_{\beta_0,\eta}, \dots ,
\nu^{p^{\beta_{l-1}}}_{\beta_{l-1},\eta}\r.
\]

Finally, we note that the forcing $\vec{\A} = \prod_{\alpha <
\eta}\A(\lambda,\kappa_\alpha)$ has a natural factorization, associated with
$F$.  Setting $\beta_{-1} = 0$ and $\beta_l = \eta$, we have \[\vec{\A} =
\prod_{n \leq l}\vec{\A}\uhr{[\beta_{n-1},\beta_{n})},\] where for each $n \leq
l$, \[\vec{\A}\uhr{[\beta_{n-1},\beta_{n})} = \prod_{\beta_{n-1} \leq \alpha <
\beta_n} \A(\lambda,\kappa_\alpha).\]  For each $0\leq n \leq l$, we define 
\[ \vec{\A}^F_{n} = i_{n}(\vec{\A}\uhr{[\beta_{n-1},\beta_{n})}) \]
and denote the resulting product by $\vec{\A}^F = \prod_{n \leq l}
\vec{\A}^F_n$.  Suppose that $\vec{H} \subseteq \vec{\A}$ is a $V$-generic
filter.  For each $n \leq l$, the fact $\vec{\A}\uhr{[\beta_{n-1},\beta_{n})}$
is a $\kappa_{\beta_{n-1}}^+$-closed forcing guarantees that
$i_n``\vec{H}\uhr{[\beta_{n-1},\beta_{n})}) \subseteq \vec{\A}^F_n$ forms a
generic filter over $N_n$, hence also $N_l=N^F$.\footnote{This is true by \cite[Proposition 15.1]{CummingsHandbook}, since the width of the embedding is smaller than its distributivity.}  For each $n$, we denote the
resulting $N_l$ generic for $\vec{\A}^F_n$ by $\vec{J}^F_n$.

This is the crucial part: the main affect of taking a non-pure extension of a condition is the reflection of its Cohen part downwards. This reflection is exactly moving from $\vec{H}$ to $\vec{J}^F_n$.

It is clear that the product $\prod_{0 \leq n \leq l} \vec{J}^F_n$ is
$\vec{\A}^F$ generic over $N_l$.  We note that for each pure $p \in \po$ with
$\vec{f}^p \in \vec{H}$, if $\vec{f}$ is the Cohen part of the natural non-pure
extension of $i^F(p)$, then $\vec{f} \in \vec{\A}^F$.  Moreover the collection
of such $\vec{f}$ forms an $\vec{\A}^F$-generic filter over $N^F$ which is
obtained by modifying $\prod_{0 \leq n \leq l} \vec{J}^F_n$ on the coordinates
$\beta_n$ for $n<l$ by taking the natural non-pure extensions of $i^F(p)$ for $p\in \vec{J}^F$, see Definition \ref{definition:natural-non-pure-extension}.  

For each individual modified filter, its genericity follows from Woodin's surgery argument
\cite{Woodin-surgery}\footnote{Indeed, for every set $a \in N^F$ of cardinality $\leq \iota_k(\kappa_{\beta_k})$, we modify the generic at most at $\kappa_{\beta_k}$ many coordinates}.  The fact that the product remains generic follows by induction on $n$ from
applications of Easton's lemma, using the chain condition of the forcing notions with smaller indexes and the closure of the components with larger indexes.  We denote by
$\vec{H}^F$ the resulting $N^F$-generic filter over $\vec{\A}^F$.

\begin{remark} \label{F-invariance} $G^*$ can be constructed in $N^F[\vec{H}^F]$
in the same way that it was constructed in $V$ using the fact that $M_\eta$ can
be described as an iterated ultrapower of $N^F$ and starting with conditions $q$
such that $s^q=F$ and $\vec{f}^q \in \vec{H}^F$. \end{remark}

We turn to the proof of Proposition \ref{Prop-G*generic}.
\begin{proof} It is clear that $G^* \subseteq j_{\eta}(\po)$ is a filter. We
verify that $G^*$ meets every dense open set $D \subseteq j_{\eta}(\po)$ in
$M_{\eta}$.  Since $M_\eta$ is the direct limit of its finite subiterates.
There are $F = \{\beta_0, \dots \beta_{l-1}\} \in [\eta]^{<\omega}$ and
$\bar{D}$ such that $k^F(\bar{D}) = D$.

Let $p'$ be the natural non-pure extension of $i^F(p)$ for some $p$ with
$\vec{f}^p \in \vec{H}$.  Now by its definition $\vec{f}^{p'} \in \vec{H}^F$.
Appealing to Lemma \ref{Lem-meetdense} and the genericity of $\vec{H}^F$, we
conclude that there exists a direct extension $p^* \geq^* p'$, with $\l
f^{p^*}_\alpha \mid \alpha <\eta\r \in \vec{H}^F$ and a finite set $\{
\alpha_0,\dots,\alpha_{m-1}\} \subset (\eta\setminus F)$, such that $p^* \fr
\vec{\nu} \in \bar{D}$ for every $\vec{\nu} = \la \nu_{\alpha_0},\dots,
\nu_{\alpha_{m-1}}\ra \in \prod_{i<m} A^{p^*}_{\alpha_i}$.  By the elementarity
of $k^F$, $k^F(p^*) \fr \vec{\nu} \in D$ for every $\vec{\nu} \in \prod_{i<m}
k^F(A^{p^*})_{\alpha_i}$.

In particular, 
  \[k^F(p^*) \fr \l \nu^{(p^*)^{\alpha_0}}_{\alpha_0,\eta}, \dots , 
  \nu^{(p^*)^{\alpha_{m-1}}}_{\alpha_{m-1},\eta}\r \in \mathcal{D}.\]
  It remains to verify that the last condition belongs to $G^*$.  This is
immediate from Remark \ref{F-invariance} and the fact that $\vec{f}^{p^*} \in
\vec{H}^F$. \end{proof}

\begin{proposition}\label{proposition: closure}
 $M_{\eta}[G^*]$ is closed under $\kappa_0 = \crit(j_{\eta})$ sequences of its elements in $V[\vec{H}]$.
\end{proposition}

\begin{proof}
Let $\l x_\mu \mid \mu < \kappa_0\r$ be a sequence of elements in
$M_{\eta}[G^*]$. Since $M_\eta[G^*]$ is a model of $\mathrm{ZFC}$, 
we may assume that all $x_\mu$ are ordinals.  By Lemma
\ref{Lem-itergens}, for each $x_\mu$ is of the form 
\[ x_\mu = j_{\eta}(g_\mu)(j_{\beta^\mu_0}(\tau^\mu_0), \dots,
j_{\beta^\mu_{l^\mu-1}}(\tau^\mu_{l^\mu-1})) \]

for some finite sequences $\beta^\mu_0 < \dots, \beta^\mu_{l^\mu-1} < \eta$ and
$\tau^\mu_0, \dots, \tau^\mu_{l^\mu-1} < \lambda$.  Moreover, since the
Rudin-Keisler order $\leq^{RK}_{E_{\beta^\mu_i}}$ is $\kappa_0^+$-directed, we
may assume that there exists some $\tau^* < \lambda$ such that $\tau^\mu_i =
\tau^*$ for all $\mu < \kappa_0$ and $i<l^\mu$. Hence for $\l x_\mu \mid \mu <
\kappa_0\r$ to be a member of $M_{\eta}[G^*]$, it suffices to verify that the
sequences $\la j_{\eta}(g_\mu) \mid \mu < \kappa_0\ra$ and $\l
j_{\alpha}(\tau^*) \mid \alpha < \eta\r$ belong to $M_{\eta}[G^*]$.

The first sequence already belongs to $M_{\eta}$ as $\crit(j_{\eta}) = \kappa_0$
implies that it is just $j_{\eta}(\vec{g})\uhr \kappa_0$.  
The latter sequence $\la j_{\alpha}(\tau^*) \mid \alpha < \eta\ra$ can be recovered from 
$G^*$ as follows. 
It follows from a simple density argument that for every $\alpha < \eta$, there exists some $q \in G^*$ such that
\begin{enumerate}
\item $\alpha$ is a non-pure coordinate of $q$, and
\item $j_{\eta}(\tau^*) \in \dom(f^q_\alpha)$.
\end{enumerate}
 It is also clear that two
conditions $q,q' \in G^*$ of this form must satisfy $f^{q}_\alpha(j_{\eta}(\tau^*)) = f^{q'}_\alpha(j_{\eta}(\tau^*))$. 
We may therefore define in $M_{\eta}[G^*]$ a function $t_{j_{\eta}(\tau^*)} \colon \eta \to \lambda^{\eta}$ by 
$t_{j_{\eta}(\tau^*)}(\alpha) = f^{q}_\alpha(j_{\eta}(\tau^*))$ for some condition $q \in G^*$ as above. 
We claim that $t_{j_{\eta}(\tau^*)}(\alpha) = j_{\alpha}(\tau^*)$ for all $\alpha < \eta$. Indeed, for every $\alpha < \eta$, 
there exists a condition $p \in \po$ with $\l f^p_\alpha \mid \alpha < \eta \r \in 
\vec{H}$ so that $\tau^* \in \dom(f^p_\alpha)$, and clearly, the condition $q = j_{\eta}(p) \fr \la \nu^{p^\alpha}_{\alpha,\eta}\ra \in G^*$ has $j_{\eta}(\tau^*) \in \dom(f^q_\alpha)$. 
But $j_{\eta}(\tau^*) \in \dom(\nu^{p^\alpha}_{\alpha,\eta}) = j_{\alpha+1,\eta}``j_{\alpha}(\dom(f^p_\alpha))$ and 
$\nu^{p^\alpha}_{\alpha,\eta} = (j_{\alpha,\eta}\uhr d_\alpha^\alpha)^{-1}$, thus, it follows that 
\[ f^{q}_\alpha(j_{\eta}(\tau^*)) = \nu^{p^\alpha}_{\alpha,\eta}(j_{\eta}(\tau^*)) = j_{\alpha}(\tau^*).\]
\end{proof}

Fix some ordinal $\beta < \eta$. The forcing $\vec{\A}$ naturally breaks into the product $\vec{\A}\uhr \beta \times \vec{\A}_{\geq \beta}$, and we observe that 
\begin{enumerate}
 \item The latter part $\vec{\A}_{\geq \beta}$ is $\kappa_\beta^+$-closed.
Therefore, if $\vec{H}_{\geq \beta} \subset \vec{\A}_{\geq \beta}$ is
$V$-generic, then by \cite[Proposition 15.1]{CummingsHandbook} its pointwise image
$j_{\beta}``\vec{H}_{\geq \beta}$ generates an $M_{\beta}$ generic filter for
the forcing \[j_{\beta}(\vec{\A}_{\geq \beta}) = \prod_{\beta \leq \alpha <
\eta} \A(\lambda^\beta,\kappa_\alpha^\beta)^{M_\beta}.\] We denote this generic
by $\vec{H}^{\beta}_{\geq \beta}$. 
 \item Using the same arguments as above, a $V$-generic filter
$\vec{H}_\beta \subset \vec{\A}\uhr\beta$ generates an $M_\beta$-generic filter
$G^*_\beta$  for the $j_{\beta}(\po_{\vec{E}\uhr \beta})$.
\end{enumerate}

We conclude that in the model $V[\vec{H}]$, we can form the generic extension
$M_\beta[G^*_\beta \times \vec{H}^\beta_{\geq \beta}]$ of $M_\beta$, with
respect to the product $j_{\beta}(\po_{\vec{E}\uhr \beta}) \times
j_{\beta}(\vec{\A}_{\geq \beta})$.

We have the following.

\begin{proposition}
 For each $\beta < \eta$, the model $M_{\beta}[G^*_\beta \times
\vec{H}^\beta_{\geq \beta}]$ can compute $M_{\eta}[G^*_{\eta}]$ and therefore \[
\bigcap_{\beta<\eta}M_\beta[G^*_\beta\times \vec{H}^\beta_{\geq\beta}] \supseteq
M_\eta[G^*].\]
\end{proposition}
It seems likely that equality holds, $\bigcap_{\beta<\eta}M_\beta[G^*_\beta\times \vec{H}^\beta_{\geq\beta}] = M_\eta[G^*]$, but we will not need that in our argument.


\section{Stationary reflection in $M_\eta[G^*_\eta]$} \label{reflection}

We assume that for every $\alpha < \eta$, there is a Laver indestructible
supercompact cardinal $\theta_\alpha$ such that $\sup_{\beta<\alpha}
\kappa_\beta < \theta_\alpha < \kappa_\alpha$.

We prove the following which completes the proof of Theorem \ref{mainthm}.

\begin{theorem} \label{gstar-reflection} In $M_{\eta}[G^*]$, every collection of
fewer than $\eta$ many stationary subsets of $j_{\eta}(\bar{\kappa}_{\eta}^+)$
reflects. \end{theorem}

We start by proving a stationary reflection fact that will be used as an
intermediate step in the proof.

\begin{claim} For every $\alpha <\eta$, every collection of fewer than $\eta$
many stationary subsets of $j_{\alpha}(\bar{\kappa}_{\eta}^+)$ with cofinalities
bounded by $\sup_{\beta<\alpha}\kappa_{\beta}^\alpha$ reflects in $M_{\alpha}[G^*_\alpha \times
\vec{H}^\alpha_{\geq\alpha}]$. \end{claim}

\begin{proof}  Let $T_i$ for $i<\mu$ be such a collection of stationary sets.
By elementarity, $j_\alpha(\theta_\alpha)$ is an indestructible supercompact
cardinal between $\sup_{\beta<\alpha}\kappa^\alpha_\beta$ and
$\kappa_\alpha^\alpha$.   Recall that $\vec{H}^\alpha_{\geq\alpha}$ is generic
for $(\kappa_\alpha^\alpha)^+$-directed closed forcing.  By the
indestructibility of $j_\alpha(\theta_\alpha)$, there is a
$j_\alpha(\bar{\kappa}_{\eta}^+)$-supercompact embedding
$k\colon M_{\alpha}[\vec{H}^\alpha_{\geq\alpha}] \to N$.  Further, $G^*_\alpha$ is
generic for $j_\alpha(\bar{\kappa}_\alpha^{++})$-cc forcing and
$j_{\alpha}(\bar{\kappa}_\alpha^{++}) < j_\alpha(\theta_\alpha)$.  By standard
arguments, we can extend $k$ to include $M_\alpha[G^*_\alpha \times
\vec{H}^\alpha_{\geq\alpha}]$ in the extension by a
$j_{\alpha}(\bar{\kappa}_\alpha^{++})$-cc forcing.  Each $T_i$ remains
stationary in this extension, so $\{k(T_i) \mid i < \mu \}$ reflects at $\sup
k``j_\alpha(\bar{\kappa}_{\eta}^+)$. 
 It follows that $\{T_i \mid i < \mu \}$ reflects in $M_\alpha[G^*_\alpha \times \vec{H}^\alpha_{\geq\alpha}]$.
\end{proof}

We begin the proof of Theorem \ref{gstar-reflection}.  Suppose that for each
$i<\mu$, $S_i \in M_{\eta}[G^*]$ is a stationary subset of
$j_{\eta}(\bar{\kappa}_{\eta}^+)$.  We can assume that for all $i$, the
cofinality of each ordinal in $S_i$ is some fixed $\gamma_i$.  It follows that
there is some $\bar{\alpha}<\eta$ such that $\sup_{i<\mu} \gamma_i <
\sup_{\beta<\bar{\alpha}}\kappa_{\beta}^{\bar{\alpha}}$.  For each $\alpha$ in the interval
$[\bar{\alpha},\eta)$, let $T_i^\alpha = \{ \delta <
j_\alpha(\bar{\kappa}_{\eta}^+) \mid j_{\alpha,\eta}(\delta) \in S_i\}$.

\begin{claim} For $\alpha \geq \bar{\alpha}$, if $\{T_i^{\alpha} \mid
i<\mu\}$ reflects at an ordinal of cofinality less than
$\kappa_\alpha^\alpha$ in $M_\alpha[G_\alpha^* \times
\vec{H}^\alpha_{\geq\alpha}]$, then $\{S_i \mid i<\mu\}$ reflects in
$M_{\eta}[G^*]$. \end{claim}

\begin{proof} Let $\delta < j_\alpha(\bar{\kappa}_{\eta}^+)$ with
$\cf(\delta) < \kappa_\alpha^\alpha$ be a common reflection point of
the collection $\{T_i^\alpha \mid i<\mu \}$.  We claim that each $S_i$ reflects at
$j_{\alpha,\eta}(\delta)$.  Let $D \subseteq
j_{\alpha,\eta}(\delta)$ be club in $M_{\eta}[G^*]$, of order type $\cf(\delta)
= \cf(j_{\alpha,\eta}(\delta))$.

Since $\crit{j_{\alpha,\eta}} > \cf(\delta)$, $j_{\alpha,\eta}$ is continuous
at $\delta$ and $E = \{\gamma < \delta \mid j_{\alpha,\eta}(\gamma) \in D
\}$ is a club in $\delta$.  Note that $E \in M_{\alpha}[G^*_\alpha \times
\vec{H}^{\alpha}_{\geq\alpha}]$.  Since $T_i^\alpha$ reflects at $\delta$,
$T_i^\alpha \cap E \neq \emptyset$ and hence $S_i \cap D \neq
\emptyset$.\end{proof}

Combining the previous two claims if for some $\alpha \geq \bar{\alpha}$,
$\{T_i^\alpha \mid i < \mu \}$ consists of stationary sets, then $\{S_i \mid
i<\mu\}$ reflects.  So we assume for a contradiction that for each
$\alpha \geq \bar{\alpha}$, there are $i_\alpha<\mu$ and a club $C_\alpha
\in M_\alpha[G_\alpha^* \times \vec{H}^\alpha_{\geq\alpha}]$ such that
$T_{i_\alpha}^\alpha \cap C_\alpha = \emptyset$.  We fix $J \subseteq \eta$
unbounded and $i^*<\mu$ such that for all $\alpha \in J$, $i_\alpha = i^*$.

Let $I_\eta$ be the ideal of bounded subsets of $\eta$.  For $\alpha \leq \eta$,
let $\vec{H}^\alpha/I_\eta$ be the generic for
$j_\alpha(\vec{\mathbb{A}})/I_\eta$ derived from $j_\alpha``\vec{H}$.  Recall
that $\vec{H}$ is generic for $\vec{\mathbb{A}}$ which is a product of Cohen
posets, so this makes sense.

\begin{claim} For $\alpha \geq \bar{\alpha}$, there is a club subset of
$C_\alpha$ in $M_\alpha[\vec{H}^\alpha/I_\eta]$.  \end{claim}

\begin{proof} We start by showing that for all $\beta > \alpha$ below $\eta$
there is a club subset of $C_\alpha$ in $M_{\alpha}[\vec{H}^\alpha \restriction [\beta, \eta)]$.  

To see this, note
that $M_\alpha[G^*_\alpha \times \vec{H}^\alpha_{\geq \alpha}]$ is
$(\bar{\kappa}^\alpha_\beta)^{++}$-cc extension of $M_\alpha[\vec{H}^\alpha
\upharpoonright [\beta,\eta)]$ and
$\bar{\kappa}_\beta^\alpha<j_\alpha(\bar{\kappa}_{\eta})$. Indeed, this forcing decomposes into a part contributing $G^*_\alpha$ which have small cardinality and the product of forcing notions $j_{\alpha}(\mathcal{A}(\lambda, \kappa_\gamma))$ for $\gamma < \beta$, whose product have $(\bar\kappa_\beta^\alpha)^{++}$-cc. For the moment
we let $\mathbb{Q}_\beta$ denote the $(\bar{\kappa}^\alpha_\beta)^{++}$-cc poset
used in order to move from $M_\alpha[\vec{H}^\alpha\restriction[\beta,\eta)]$ to $M_\alpha[G_\alpha^* \times \vec{H}^\alpha_{\geq \alpha}]$.

We set take $C_{\alpha,\beta}$ to be the set of closure points of the function
assigning each $\gamma$ to the supremum of the set 
\[\{\gamma^* \mid \exists q\in \mathbb{Q}_\beta,\, q \Vdash \check{\gamma}^* = \min (\dot{C}_\alpha \setminus (\check\gamma + 1))\},\]
as computed in the model
$M_\alpha[\vec{H}^\alpha \upharpoonright [\beta,\eta)]$.  Clearly this is a
club.  Further if $\beta < \beta'$, then $C_{\alpha,\beta'} \subseteq
C_{\alpha,\beta}$.

Since the clubs are decreasing, $\bigcap_{\beta > \alpha} C_{\alpha,\beta}$ is
definable in $M_\alpha[\vec{H}^\alpha/I_\eta]$ as the set of ordinals $\gamma$ 
such that for some condition $\vec{a} \in j_\alpha(\vec{\mathbb{A}})$, $\vec{a}/I_\eta \in \vec{H}^\alpha / I_\eta$ and for all sufficiently
large $\beta$, $\vec{a} \upharpoonright [\beta,\eta)$ forces $\gamma \in
C_{\alpha,\beta}$.  So $\bigcap_{\beta>\alpha}C_{\alpha,\beta}$ is as required
for the claim. \end{proof}

Let $\dot{D}_\alpha$ be a $\vec{H}^\alpha/I_\eta$-name for the club from the
previous lemma.

\begin{lemma}
$\vec{H}^{\eta}/I_\eta \in M_{\eta}[G^*]$.
\end{lemma}
\begin{proof}
Let $p \in j_\eta(\vec(\mathbb{A}))$. We will show how to decide in $M_\eta[G^*]$ whether $p \in \vec{H}^\eta/I_\eta$ or not.

Now, fix $q \in G^*$ with $\dom q(\alpha) \supseteq \dom p(\alpha)$. Let us claim that $p \in \vec{H}/I_\eta$ if and only if there is $\beta$ such that $f^q(\gamma)\supseteq p(\gamma)$ for all $\gamma \in [\beta, \eta)$. 

Indeed, without loss of generality, $q = j_\eta(\bar{q}) ^\smallfrown \langle \nu_0,\dots, \nu_{m-1}\rangle$ for $\bar{q} \in \vec{H}$ and $\nu_0,\dots, \nu_{m-1}$ generators.

Thus, for every $\beta$ which is larger than the index of the measures of $\nu_0,\dots,\nu_{m-1}$ $f^{\bar{q}(\beta)} \in H(\beta)$ and thus after applying $j_\eta$ we conclude that the condition which is defined by $\langle f^q(\gamma) \mid \beta \leq \gamma < \eta\rangle$ belongs to $\vec{H}^\eta \restriction [\beta,\eta)$. Since it is compatible with $p$, $p \in \vec{H}^\eta/I_\eta$.

In the other direction, one can take $\beta$ large enough so that $p \restriction [\beta,\eta) \in \vec{H}^\eta \restriction [\beta,\eta)$ and then $p \restriction [\beta,\eta)$ is the Cohen part of a condition in $G^*$.    
\end{proof}
\begin{claim} $\langle
j_{\alpha,\eta}(\dot{D}_\alpha)_{\vec{H}^{\eta}/I_\eta} \mid \bar{\alpha} \leq
\alpha < \eta \rangle \in M_{\eta}[G^*]$. \end{claim}

\begin{proof}  By the previous lemma,
$\vec{H}^{\eta}/I_\eta \in M_{\eta}[G^*]$.  Further by Proposition
\ref{proposition: closure}, $M_{\eta}[G^*]$ is closed under $\eta$-sequences.  So
the sequence of names for clubs $\langle j_{\alpha,\eta}(\dot{D}_\alpha) \mid \alpha < \eta \rangle \in M_{\eta}[G^*]$.   \end{proof}

To get a contradiction it is enough to show that $\bigcap_{\bar{\alpha} \leq
\alpha<\eta} j_{\alpha}(\dot{D}_\alpha)_{\vec{H}^{\eta}/I_\eta} \cap
S_{i^*} = \emptyset$.  Suppose that there is some $\delta$ in the intersection.
We can find some $\alpha\in J$ and $\bar{\delta}$ such that
$j_{\alpha,\eta}(\bar{\delta}) = \delta$.  However, by the definitions of
$D_\alpha$ and $T_{i^*}^\alpha$, we must have that $\bar{\delta} \in C_\alpha
\cap T_{i^*}^\alpha$, a contradiction.

This completes the proof of Theorem \ref{mainthm}.

\section{Bad scales} \label{badscale}

In this section we give the proof of Theorem \ref{mainthm2}.  For this theorem
we work with the forcing $\po$ as before and assume that there is an
indestructibly supercompact cardinal $\theta < \kappa_0$.  Working in $V$, let
$\vec{f}$ be a scale of length $\bar{\kappa}_\eta^+$ in
$\prod_{\alpha<\eta}\kappa_\alpha^+$.  As before let $\vec{H}$ be
generic for $\vec{\A}$ over $V$ and let $G^*$ be the $j_\eta(\po)$ generic over
$M_\eta$ defined above.

\begin{lemma} In $M_\eta[G^*]$, $j_\eta(\vec{f})$ is a scale of length
$j_\eta(\bar{\kappa}_\eta^+)$ in $\prod_{\alpha<\eta}j(\kappa_\alpha)^+$.
\end{lemma}

\begin{proof} Let $g \in \left(\prod_{\alpha<\eta}j(\kappa_\alpha)^+ \right)\cap
M_\eta[G^*]$.  Clearly $g \in V[\vec{H}]$.  Since each $\kappa_\alpha^+$ is a
continuity point of $j_\eta$, we can find an ordinal
$\gamma_\alpha<\kappa_\alpha^+$ such that $j_\eta(\gamma_\alpha)>g(\alpha)$.

By the distributivity of $\vec{\mathbb{A}}$, the sequence 
$\tilde{g} = \langle \gamma_\alpha \mid \alpha < \eta\rangle$ 
belongs to $V$. Pick (in $V$) an ordinal $\zeta$ such that $f_\zeta$ dominates $\tilde{g}$. Then, 
$j_\eta(f_\zeta) = j_\eta(f)_{j_\eta(\zeta)}$ dominates $g$.
\end{proof}

\begin{lemma} For $\delta<\kappa_\eta^+$ with $\eta < \cf(\delta)<\kappa_0$, if
$j_\eta(\delta)$ is a good point for $j(\vec{f})$ in $M_\eta[G^*]$, then
$\delta$ is a good point for $\vec{f}$ in $V[\vec{H}]$. \end{lemma}

\begin{proof}  Suppose that $A \subseteq j_\eta(\delta)$ and $\alpha^* < \eta$
witness that $j_\eta(\delta)$ is good for $\vec{f}$ in $M_\eta[G^*]$.   Note
that $j_\eta``\delta$ is cofinal in $j(\delta)$.  By thinning $A$ if necessary
we let $B \subseteq j_\eta``\delta$ be an unbounded such that each element
$\gamma \in B$ has a greatest element of $A$ less than or equal to it.  For each
$\gamma$ in $B$, let $\alpha_\gamma$ be such that for all $\alpha \geq
\alpha_\gamma$, \[j(\vec{f})_{\max(A \cap (\gamma+1))}(\alpha) \leq
j(\vec{f})_\gamma(\alpha) < j(\vec{f})_{\min(A \setminus (\gamma+1))}(\alpha).\]
Let $B'$ be an unbounded subset of $B$ on which $\alpha_\gamma$ is fixed.  Then
$B'$ witnesses that $j_\eta(\delta)$ is good.  Since $B' \subseteq
j_\eta``\delta$, we have that $\{ \gamma < \delta \mid j_\eta(\gamma) \in B' \}$
witnesses that $\delta$ is good for $\vec{f}$ in $V[\vec{H}]$. \end{proof}

In $V[\vec{H}]$, let $S = \{ \delta < \bar{\kappa}_\eta^+ \mid \delta$ is
nongood for $\vec{f}$ and $\eta< \cf(\delta) < \theta \}$.  We claim that $S$ is
stationary.  Let $k\colon V[\vec{H}] \to N$ be an elementary embedding witnessing
that $\theta$ is $\bar{\kappa}_\eta^+$-supercompact in $V[\vec{H}]$.  Standard
arguments show that $\sup k``\bar{\kappa}_\eta^+$ is a nongood point for
$k(\vec{f})$.  It follows that $S$ is stationary, since $\sup
k``\bar{\kappa}_\eta^+ \in k(C)$ every club $C \subseteq \bar{\kappa}_\eta^+$ in
$V[\vec{H}]$.

Now suppose that in $M_\eta[G^*]$, there is a club $D$ of good points for
$j_\eta(\vec{f})$.  In $V[\vec{H}]$, let $C = \{ \delta < \bar{\kappa}_\eta^+
\mid j_\eta(\delta) \in D \}$.  By the previous lemma, $C$ is a $<
\kappa_0$-club consisting of good points for $\vec{f}$.  However, $S \cap C$ is
nonempty, a contradiction.

\section{Countable cofinality}\label{ctble}

Our goal in this section is to show that given the assumptions of Theorem \ref{mainthm3}, 
there exists a forcing extension which adds a generic filter for the extender based Prikry
forcing by a $(\kappa,\kappa^{++})$-extender (in particular, forces $\cf(\kappa)
= \omega$  and $2^\kappa = \kappa^{++}$) and satisfies that every stationary subset of $\kappa^+$ reflects. 

A necessary step for obtaining the latter is to add a club $D \subseteq \kappa^+$ disjoint from $S^{\kappa^+}_{\kappa}$, the set of ordinals $\alpha < \kappa^+$ of cofinality $\kappa$ in the ground model. 
In \cite{HU}, the second and third authors address the situation for adding a single Prikry sequence to $\kappa$. It is shown that under the subcompactness assumption of $\kappa$, there is a Prikry-type forcing which both singularizes $\kappa$ and adds a club $D$ as above, without generating new nonreflecting stationary sets. 

An additional remarkable aspect of the argument of \cite{HU} is that it presents a fertile framework in which the arguments address an extension $N_\omega[\mathcal{H}]$ of an iterated ultrapower $N_\omega$ of $V$, and assert directly that stationary reflection holds in $N_{\omega}[\mathcal{H}]$ without having to specify the poset by which $\mathcal{H}$ is added. 

Let us briefly describe the key ingredients of the construction of \cite{HU}, to serve as a reference for our
arguments in the context of the failure of SCH.
In \cite{HU}, one starts from a normal measure $U$ on $\kappa$ in $V$, and consider the $\omega$-iterated ultrapower by $U$, given by 
\[N_0 = V\text{, }i_{n,n+1} \colon N_n \to N_{n+1} \cong \Ult(N_n,i_n(U)),\]and the
direct limit embedding $i_{\omega} \colon V \to N_\omega$. 
It is well known that the sequence of critical points 
$\l \kappa_n \mid n < \omega\r$, $\kappa_{n+1} = i_{n,n+1}(\kappa_{n})$ is
Prikry generic over $N_\omega$ and that $N_{\omega}[\l \kappa_n \mid n<\omega
\r] = \bigcap_n N_n$, (\cite{Bukovski, Dehornoy}).  

Let $\qo$ be the Prikry name for the forcing for adding a disjoint club
from $(S^{\kappa^+}_\kappa)^V = \kappa^+ \cap \cof^V(\kappa)$. The forcing $\qo_{\omega} = i_\omega(\qo)^{\l \kappa_n \mid n < \omega\r} \in N_\omega[{\l \kappa_n \mid n < \omega\r}]$ is isomorphic to the forcing for adding a $\kappa^+$-Cohen set over $V$, and likewise, to adding a $\kappa_n^+$-Cohen  set  over $N_n$, for each $n < \omega$. 

Moreover, taking a $\qo_{\omega}$-generic filter $H $ over $V$, we have that both $H$ and $i_{0,1}``H $ generate mutually generic filters
over $N_1$ for $i_1(\qo_\omega) = \qo_\omega$.
More generally, for each $n$, the sequence  $H^n_0, \dots, H^n_n$, where each $H^n_k$ is generated by $i_{k,n}``H \subset i_n(\qo_\omega) = \qo_\omega$ for $1 \leq k \leq n$, are mutually generic filters for $\qo_\omega$ over $N_n$. With this choice of ``shifts'' of $H$, we obtain that for each $n < k$, the standard iterated ultrapower map $i_{n,k} \colon N_n \to N_k$ extends to 
$i_{n,k}^* \colon N_n[\l H^n_0,\dots, H^n_n \r] \to N_k[ \l H^{k}_0, \dots, H^{k}_n\r] \subseteq  N_k[ \l H^{k}_0, \dots, H^{k}_k\r]$.
For each $n$, the sequence $\l H^n_0, \dots, H^n_n \r$ is denoted by $\mathcal{H}_n$. 
The final extension $N_\omega[\mathcal{H}]$ of $N_\omega$ 
is given by the sequence $\mathcal{H} = \l H_n^\omega \mid n < \omega\r$, were $H_n^\omega$ is the filter generated by $i_{n,\omega}``H$, which achieves the critical  equality \[N_\omega[\mathcal{H}] = \bigcap_n N_n[\mathcal{H}_n].\] 

From this equality it follows at once that:
\begin{enumerate}
\item $N_\omega[\mathcal{H}]$ is closed under its $\kappa$-sequences;
\item $\kappa_\omega = i_\omega(\kappa)$ is singular in $N_\omega[\mathcal{H}]$,
as $\l \kappa_n \mid n<\omega \r$ belongs to each $N_n[\mathcal{H}_n]$;  and
\item $H \in N_\omega[\mathcal{H}]$, as $H = H^n_n$ for all $n < \omega$. 
\end{enumerate}

Therefore every stationary subset $S$ of $\kappa_\omega^+$ in $N_\omega[\mathcal{H}]$ can be assumed to concentrate at some cofinality $\rho < \kappa_\omega$. Say for simplicity that $\rho < \kappa_0$, 
one shows that $S$ reflects in $N_\omega[\mathcal{H}]$ by examining its pull
backs $S_n = i_{n,\omega}^{-1}(S) \subseteq \kappa_n^+$. If we have that
$\kappa_n = i_n(\kappa)$ is $\Pi^1_1$-subcompact in the Cohen extension
$N_n[\mathcal{H}_n]$ of $N_n$, then we have that if $S_n$ is stationary then it must reflect at some $\delta < \kappa_n^+$ of cofinality $\delta  <\kappa_n$.  Using $i_{n,\omega}$, we conclude from this that $S$ reflects at $i_{n,\omega}(\delta)$ in $N_\omega[\mathcal{H}]$. 

To rule out the other option, of having all $S_n \subseteq \kappa_n^+$ being
nonstationary in $N_n[\mathcal{H}_n]$, one takes witnessing disjoint clubs $C_n
\subseteq \kappa_n^+$ and uses the fact that for each $n < \omega$,
$i_{n,\omega} \colon N_n \to N_\omega$ extends to $i_{n,\omega}^* \colon
N_n[\mathcal{H}_n] \to N_\omega[  i_{n,\omega}``\mathcal{H}_n] \subset
N_{\omega}[\mathcal{H}]$. This allows us to show that the club $D_n =
i_{n,\omega}^*(C_n)  \subseteq \kappa_\omega^+$ belongs to
$N_\omega[\mathcal{H}]$ for each $n$. Since $N_\omega[\mathcal{H}]$  is closed
under its $\kappa$ sequences, it computes $\l D_n \mid n < \omega\r$ correctly
and thus also $D = \bigcap_n D_n$, that would have to be disjoint from $S$, a
contradiction.

We turn now to the new construction and prove Theorem \ref{mainthm3}.  Let $V'$ be a model which contains a $\kappa^+$-$\Pi^1_1$-subcompact cardinal $\kappa$, which also carries a $(\kappa,\kappa^{++})$-extender. 

Recall that $\kappa$ is $\kappa^+$-$\Pi^1_1$-subcompact if for every set $A \subseteq H(\kappa^+)$ and every $\Pi^1_1$-statement $\Phi$ such that 
$\l H(\kappa^+), \in ,A\ra \models \Phi$, there are $\rho < \kappa$, $B \subseteq H(\rho^+)$, and an elementary embedding 
$j \colon \l H(\rho^+), \in ,B\ra \to \l H(\kappa^+), \in ,A\r$ with $cp(j) = \rho$, such that $\l H(\rho^+), \in ,B\ra \models \Phi$.

Let $V$ be obtained from $V'$ by an Easton-support iteration of products $\Add(\alpha^+,\alpha^{++})$ for inaccessible $\alpha\leq\kappa$.  

By 
\cite[Lemma 42]{HU}, $\kappa$ remains $\kappa^+$-$\Pi_1^1$-subcompact in $V$ and even
in a further extension by $\Add(\kappa^+,\kappa^{++})$.  Moreover, by standard
argument it is routine to verify that $\kappa$ still carries a
$(\kappa,\kappa^{++})$-extender in $V$.
We note that as a consequence of $\kappa^+$-$\Pi_1^1$-subcompactness in $V$,
simultaneous reflection holds for collections of fewer than $\kappa$ many
stationary subsets of $S^{\kappa^+}_{<\kappa}$.  Further this property is
indestructible under $\Add(\kappa^+,\kappa^{++})$.

Working in $V$, let $E$ be a $(\kappa,\kappa^{++})$-extender.  Let \[\langle
j_{m,n}\colon M_m \to M_n \mid m \leq n \leq \omega \rangle\] be the iteration by $E$
and \[\langle i_{m,n}\colon N_m \to N_m \mid m \leq n \leq \omega\rangle\] be the iteration by the
normal measure $E_\kappa$ where $V = M_0 = N_0$.
We write $j_n$ for $j_{0,n}$ and $i_n$ for $i_{0,n}$.

We describe a generic extension of $M_\omega$ in which $j_\omega(\kappa^+)$
satisfies the conclusion of the theorem.  It follows by elementarity that there
is such an extension of $V$.  We are able to isolate the forcing used, but this
is not required in the proof.

We start by constructing a generic object for 
$j_\omega(\mathbb{P}_E)$ over $M_\omega$, where $\mathbb{P}_E$ is the 
 extender based forcing of Merimovich. Although for the most part, we will refer to Merimovich's arguments in \cite{Merimovich}, our presentation follows a more up-do-date presentation of the forcing, given by Merimovich in 
 \cite[Section 2]{Merimovich2}.

We recall that conditions $p \in \po_E$ are pairs of the form $p = \l f, T\r$, where $f \colon  d \to [\kappa]^{<\omega}$ is a partial function from $\kappa^{++}$ to 
$[\kappa]^{<\omega}$ with domain $d \in [\kappa^{++}]^{\leq \kappa}$ with $\kappa \in d$,  and $T$ is tree of height $\omega$, whose splitting sets are all measure one with respect to a a measure $E(d)$ on $V_\kappa$, derived from the extender $E$. The generator of $E(d)$ is the function $\mc(d) = \{ \l j(\alpha), \alpha \r \mid \alpha \in d\}$. 
Therefore, a typical node in the tree is an increasing sequence of functions $\l \nu_0,\dots,\nu_{k-1}\r$ where each $\nu_i$ is a partial, order preserving function $\nu_i \colon  \dom(\nu_i) \to \kappa$, with $\kappa \in \dom(\nu_i)$ and $|\nu_i| = \nu_i(\kappa)$. The sequence $\l \nu_0,\dots, \nu_{k-1}\r$ is increasing in the sense that $\nu_i(\kappa) < \nu_{i+1}(\kappa)$ for all $i$. 
When extending conditions $p = \l f, T\r \in \po_E$ we are allowed to 
\begin{enumerate}
\item[(i)] extend $f$ in as Cohen conditions (namely, add points $\gamma < \kappa^{++}$ to $\dom(f)$ and  arbitrarily choose $f(\gamma) \in [\kappa]^{<\omega}$), 
and modify the tree parts;
\item[(ii)] shrink the tree $T$; and  
\item[(iii)] choose a point 
$\l \nu \r \in \succ_{\emptyset}(T)$ to extend $p = \l f,T\r$ to $p_{\l \nu \r} = \l f_{\l \nu \r}, T_{\l \nu \r}\r$, where $f_{\l \nu \r}$ is defined by 
$\dom(f_{\l \nu \r}) = \dom(f)$ and 
$$
f_{\l \nu \r}(\alpha)  = 
\begin{cases}
 f(\alpha) \cup \{\nu(\alpha)\}  &\mbox{if }  \alpha \in \dom(\nu) \text{ and } \nu(\alpha) > \max(f(\alpha))\\
 f(\alpha) &\mbox{otherwise }
 \end{cases}
$$
\end{enumerate}

Any extension of $p$ is obtained by finite combination of (i)-(iii), and
the direct extensions of $p$ are those which are obtained by (i),(ii). 
The poset $\po_E^*$ is the suborder of $\po_E$ whose extension consists only of the Cohen type extension (i).
Clearly, $\po_E^*$ is isomorphic to $\Add(\kappa^+,\kappa^{++})$.

We turn to describe the construction of a $j_\omega(\po_E)$ generic following \cite{Merimovich}.
Let $G_0$ be $\mathbb{P}_E^*$ generic.  
We work by induction to define $G_n$ for
$n < \omega$.  Suppose that we have defined 
$G_n \subseteq j_n(\mathbb{P}_E^*)$
for some $n<\omega$.  First, let $G_{n+1}'$ be the closure of the set 
$j_{n,n+1}``G_n$. Then, we take $G_{n+1}$ to be obtained from $G_{n+1}'$ by adding
the ordinal $\alpha$ to $G_{n+1}'$
at coordinate $j_{n,n+1}(\alpha)$, for all $\alpha < j_n(\kappa^{++})$ .

\begin{claim}  $G_{n+1}$ is $M_{n+1}$-generic for $j_{n+1}(\mathbb{P}_E^*)$.
\end{claim}

\begin{proof} This is a straightforward application of Woodin's surgery argument
\cite{Woodin-surgery}, so we only sketch the proof.  Let $D$ be a dense open
subset of $j_{n+1}(\mathbb{P}_E^*)$.  Let $E$ be the set of all $f$ in $D$ such
that all $j_{n}(\kappa)$ sized modifications of $f$ are in $D$.  $E$ is dense
using the closure of $j_{n+1}(\mathbb{P}_E^*)$.  Now the fact that $G_{n+1}'$
meets $E$ implies that $G_{n+1}$ meets $D$, since each condition in $G_{n+1}$ is
a $j_n(\kappa)$ sized modification of one in $G_{n+1}'$. \end{proof}

We make a few remarks.

\begin{remark} \label{combined-alterations} For each $n$, $G_{n}$ can be obtained directly from the upwards
closure of $j_n``G_0$ by combining the alterations used in construction of 
$G_i$
for $i<n$. \end{remark}

Since the proof of the previous claim can be repeated in any suitably closed
forcing extension, we have the following.

\begin{remark} \label{H-mutually-generic}  If $H$ is generic for
$j_n(\kappa^+)$-closed forcing and mutually generic to the upwards closure of
$j_n``G_0$, then $H$ and $G_n$ are mutually generic. \end{remark}

Let $G_\omega$ be the $j_\omega(\mathbb{P}_E)$-generic obtained from $G_0$ as in
\cite{Merimovich}.

\begin{claim} $M_\omega[G_\omega]$ is closed under $\kappa$-sequences.
\end{claim}

\begin{proof} It is enough to show that $M_\omega[G_\omega]$ is closed under
$\kappa$-sequences of ordinals.  Let $\l \gamma_\delta \mid \delta < \kappa \r$
be a sequence of ordinals.  We can assume that each $\gamma_\delta$ is of the
form $j_\omega(g_\delta)(\alpha_\delta, j(\alpha_\delta),\dots
j_{n_\delta-1}(\alpha_\delta))$ for some $g_\delta\colon[\kappa]^{n_\delta} \to
\kappa^{++}$ and $\alpha_\delta<\kappa^{++}$.  We refer the reader to
\cite{Merimovich} Corollary 2.6 for a proof.  Since $\l
j_\omega(g_\delta) \mid \delta < \kappa \r = j_\omega(\l g_\delta \mid
\delta<\kappa\r)\uhr \kappa \in M_\omega$, it is enough to show that $\l
(\alpha_\delta, \dots j_{n_\delta-1}(\alpha_\delta)) \mid \delta < \kappa \r \in
M_\omega[G_\omega]$.  To see this note that $\{\alpha_\delta \mid \delta <
\kappa\} = \dom(f)$ for some $f \in G_0$ and that the sequence $f(\alpha_\delta)
\frown (\alpha_\delta,
\dots j_{n_\delta-1}(\alpha_\delta))$ is an initial segment of the
$\omega$-sequence with index $j_\omega(\alpha_\delta)$ in $G_\omega$.  This is
enough to compute $\l
(\alpha_\delta, \dots j_{n_\delta-1}(\alpha_\delta)) \mid \delta < \kappa \r$,
since $j_\omega ``f \in M_\omega$. \end{proof}

The following lemma is proved in \cite[Section 4]{HayutBD}. 

\begin{lemma} \label{intersection1}  $\bigcap_{n<\omega} M_n[G_n] =
M_\omega[G_\omega]$. \end{lemma}
Since we will need to prove a stronger version ahead, let us sketch the proof of this lemma.
\begin{proof}
First, the inclusion $\bigcap_{n<\omega} M_n[G_n] \supseteq
M_\omega[G_\omega]$ is clear. 

As in the case of the standard Bukovsk\'y-Dehornoy Theorem, in order to prove the other direction we would like to take a set of ordinals $X \in M_\omega[G_\omega]$ and consider the sets
\[Y_n = \{\alpha \mid j_{n,\omega}(\alpha) \in X\} \in M_n[G_n].\]

Note that the filter generated by $j_{n,\omega}`` G_n$ is not a member of $M_\omega[G_\omega]$ and thus we must be more careful above the way we push $Y_n$ into $M_\omega[G_\omega]$. 

Let us pick for each $n$ a name $\dot{y}_n \in M_n$ for $Y_n$. As the sequence of names $\langle j_{n,\omega}(\dot{y}_n)\mid n < \omega\rangle$ is a member of $M_\omega[G_\omega]$ (using its closure under countable sequence), we must verify that we can evaluate correctly, without access to the filters $j_{n,\omega}(G_n)$, whether an ordinal $\gamma$ belongs to a tail of sets $j_{n,\omega}(\dot{y}_n)_{j_{n,\omega}(G_n)}$. 

This is done using Merimovich's genericity criteria, \cite{MerimovichGenCriteria}. Indeed, for $a \subseteq j_{\omega}(\kappa^{++})$ set of ordinals of cardinality $j_{\omega}(\kappa)$, which is the set of ordinals of an elementary substructure of a sufficiently large initial segment of $M_\omega$, one can consider a tree Prikry type forcing using $j_\omega(E)(a)$, and verify that there is a generic sequence $\langle t_n \mid n < \omega\rangle$ for it in $M_\omega[G_\omega]$ such that,
$\exists q \in j_\omega(\mathbb{P}_E^*)$ such that for all $n$, $q^\smallfrown\langle t_0,\dots, t_{n-1}\rangle \in G_\omega$. Moreover, this sequence is unique up to shifts and finite modifications, see \cite[Lemma 3.4]{MerimovichGenCriteria} --- which is the crux of this argument. 

As $j_{\omega}(\kappa) \in a$ and we may assume that the Prikry sequence at this coordinate is the standard one, one can fix the shift and identify the sequence $\langle j_{\omega}(p)^\smallfrown \langle \nu_0, \dots, \nu_{n-1}\rangle \mid n < \omega\rangle$ up to a finite error. 

In particular, by taking $a$ to be large enough, we conclude that there is $p \in G_0$ such that for all large $n$,
\[j_{\omega}(p)^\smallfrown \langle \nu_0, \dots, \nu_{n-1}\rangle \Vdash\check{\gamma} \in j_{n,\omega}(\dot{y}_n)\]
if and only if $\gamma \in X$.

Therefore, $\gamma \in X$ if and only if there is a condition $q\in G_\omega$ with large domain $a$  and a Prikry generic $\langle t_n \mid n < \omega\rangle$ such the length of $q(j_\omega(\kappa))$ is zero and for all large $n$  
\[q^\smallfrown \langle \nu_0, \dots, \nu_{n-1}\rangle \Vdash\check{\gamma} \in j_{n,\omega}(\dot{y}_n).\]
\end{proof}



We are now ready to describe the generic extension of $M_\omega$.  We recall
some of the basic ideas from \cite{HU}.  Let $\dot{\mathbb{Q}}$ be
the canonical name in the Prikry forcing defined from $E_\kappa$ for the poset
to shoot a club disjoint from the set of $\alpha < \kappa^+$ such that
$\cf^V(\alpha) = \kappa$.  Let
$\mathbb{Q}_\omega$ be the forcing $i_\omega(\dot{\mathbb{Q}})$ as interpreted
by the critical sequence $\langle i_n(\kappa) \mid n<\omega \rangle$.  By the
argument following Claim 39 of \cite{HU}, $\mathbb{Q}_\omega$ is
equivalent to the forcing $\Add(\kappa^+,1)$ in $V$.  In fact, for every
$n<\omega$, in $N_n$ it is equivalent to $\Add(i_n(\kappa^+),1)$.

Let $k_n\colon N_n \to M_n$ be the natural embedding given by 
\[
k_n(i_n(f)(\kappa,
\dots i_{n-1}(\kappa))) = j_n(f)(\kappa, \dots j_{n-1}(\kappa)).
\]

We discussed the notion of a \emph{width} of an elementary embedding in the previous section. We recall a fundamental result concerning  lifting embeddings and their widths (see \cite{CummingsHandbook}).

\begin{lemma}\label{Lem:widthgen}
 Suppose that $k \colon  N \to M$ has width $\mu$, $\qo \in N$ is a $\mu^+$-distributive forcing, and $H \subseteq \qo$ is generic over $N$, then $k``H \subset k(\qo)$ generates a generic filter 
 $\l k``H  \r$ for $M$.
\end{lemma}

\begin{claim}\label{claim:bounding-the-width-of-k} For all $1\leq n<\omega$, $k_n$ has width $\leq
(i_{n-1}(\kappa)^{++})^{N_n}$. \end{claim}

\begin{proof}  Let $x = j_n(g)(\alpha,j(\alpha), \dots j_{n-1}(\alpha))$ be an
element of $M_n$, $g\colon [\kappa]^n \to V$ and $\alpha < \kappa^{++}$.
Consider $h = i_n(g) \uhr (i_{n-1}(\kappa)^{++})^{N_n}$.  Then 
\[k(h)(\alpha,j(\alpha), \dots j_{n-1}(\alpha)) = x.\] 
Indeed, the domain of $k(h)$ is large enough for this evaluation to make sense. \end{proof}

We are now ready to define the analogs of $\mathcal{H}$ and $\mathcal{H}_n$ for
our situation.  Recall that for subset $X$ of some poset, we write $< X >$ for
the upwards closure in that poset.  It will be clear from the
context which poset we are working with, when taking this upwards closure. 

Let $H$ be generic for $\mathbb{Q}_\omega$ over $V[G_0]$ and recall
that $\qo_\omega = i_n(\qo_\omega)$ is also a member of $N_n$ for every $n$, and is $i_n(\kappa^+)$-closed. We note that clearly, $i_n(\kappa^+) > i_{n-1}(\kappa^{++}) \geq \text{width}(k_n)$.  

Let 
\[\mathcal{H}_n
= \langle {<}j_{m,n} \circ k_m``H{>} \mid m \leq n \rangle\] 
and 
\[\mathcal{H} = \langle
{<}j_{m,\omega} \circ k_m``H{>} \mid m < \omega \rangle.\] 
 Note that 
$\mathcal{H}_n$ is a subset of $\prod_{m \leq n} j_{m,n}(k_m(\qo_\omega))$.  We
do not yet know that it is generic.

We prove a sequence of claims about $\mathcal{H}_n$ and $\mathcal{H}$.  The
first is straightforward.

\begin{claim} For all $n<\omega$, \[\mathcal{H}_{n+1} = {<}j_{n,n+1} ``
\mathcal{H}_n{>} \frown {<}k_{n+1}``H{>}.\] \end{claim} 

Let $\bar{\mathcal{H}}_n = \langle {<} i_{m,n}``H{>} \mid m \leq n \rangle$.

\begin{claim} $\mathcal{H}_n = {<} k_n `` \bar{\mathcal{H}}_n{>}$.  \end{claim}

This is immediate from the following.  For $m \leq n$, we have \[{<}k_n ``
{<}i_{m,n}``H{>}{>} = {<}j_{m,n} `` {<}k_m `` H{>}{>}\] since $k_n \circ i_{m,n} = j_{m,n}
\circ k_m$.  In particular, since $\bar{\mathcal{H}}_n$ is generic for
$i_n(\kappa^+)$-closed forcing over $N_n$ and $k_n$ has width less than
$i_n(\kappa)$, $\mathcal{H}_n$ is generic over $M_n$.

\begin{claim} $\mathcal{H}_n$ is mutually generic with $G_n$ over $M_n$.
\end{claim}

\begin{proof} Since $G_0$ and $H$ are mutually generic over $V$, we can repeat
the argument from Claim 18 of \cite{HU} to see that $\bar{\mathcal{H}}_n$ is
generic over the model $N_n[{<}i_n``G_0{>}]$.  By the previous claim, ${<}j_n``G_0{>}$ and
$\mathcal{H}_n$ are mutually generic over $M_n$.  By Remark
\ref{H-mutually-generic}, $\mathcal{H}_n$ is mutually generic with $G_n$.
\end{proof}

\begin{lemma} $M_\omega[G_\omega][\mathcal{H}]$ is closed under
$\kappa$-sequences. \end{lemma}

\begin{proof} It is enough to show that every $\kappa$-sequence of ordinals from the larger model
$M_\omega[G_\omega][\mathcal{H}]$ is in $M_\omega[G_\omega]$.  Let
$\vec{\alpha}$ be such a sequence.  By construction, $\vec{\alpha} \in
V[G_0][H]$ and hence in $V[G_0]$.  However, $M_\omega[G_\omega]$ is closed under
$\kappa$-sequences in $V[G_0]$, so $\vec{\alpha} \in M_\omega[G_\omega]$, as
required. \end{proof}

\begin{lemma} \label{intersection2} $\bigcap_{n<\omega}M_n[G_n][\mathcal{H}_n] =
M_\omega[G_\omega][\mathcal{H}]$. \end{lemma}

\begin{proof} The proof is similar to the proof of Lemma \ref{intersection1}.
First, note that the embedding $j_{n,\omega}\colon M_n \to M_\omega$ lifts to an elementary embedding
\[j^*_{n,\omega}\colon M_n[\mathcal{H}_n] \to M_\omega[{<}j_{n,\omega}``\mathcal{H}_n{>}],\] and
$M_\omega[{<}j_{n,\omega}``\mathcal{H}_n{>}] \subseteq M_{\omega}[\mathcal{H}]$.
Further, the construction of $G_\omega$ is the same whether we start in $M_n$ or
$M_n[\mathcal{H}_n]$.

It is immediate that $\bigcap_{n<\omega}M_n[G_n][\mathcal{H}_n] \supseteq
M_\omega[G_\omega][\mathcal{H}]$.  For the other inclusion, we work as before
and suppose that $x \in \bigcap_{n<\omega}M_n[G_n][\mathcal{H}_n]$ is a set of
ordinals.  For each $n$, we can define $y_n = \{ \gamma \mid
j_{n,\omega}(\gamma) \in x \}$.  Continuing as before, but working in
$M_n[\mathcal{H}_n]$, we have that $y_n = \dot{y}_n^{G_n}$ for some name
$\dot{y}_n$.


By the previous lemma, $M_\omega[G_\omega][\mathcal{H}]$ has the sequence $\l
j^*_{n,\omega}(\dot{x}_n^*) \mid n<\omega \r$. 
As in the proof of Lemma \ref{intersection1}, we will now argue that for every suitable set of ordinals $a \subseteq j_\omega(\kappa^{++})$ of cardinality $j_\omega(\kappa)$, there is a unique (up to finite shift and finite modification) sequence $\langle t_n \mid n < \omega\rangle$ which is generic for the Prikry forcing for the measure $j_\omega(E)(a)$ and satisfies that $q ^\smallfrown \langle t_n \mid n < \omega\rangle \in G_\omega$ for all $n$. 

Since there is $p\in G$ such that for $j_{\omega}(p)$ the sequence $\langle \nu_n \restriction a \mid n < \omega\rangle$ will be generic for the corresponding forcing, we conclude (by uniqueness) that for every such sequence, for all large $n$, \[q ^\smallfrown \langle t_n \mid n < \omega\rangle \in j_{n,\omega}(G_n)\]   
assume that the length of $q$ at $j_{\omega}(\kappa)$ is zero, so there is no shift relative to the indexing of the $G_n$-s.

For each ordinal $\gamma$ one can take $a$ large enough so that for all $n$, the truth value of $\gamma \in j_{n,\omega}(\dot{y}_n)$ is decided by $q ^\smallfrown \langle t_i \mid i < n\rangle$. Therefore, $\gamma \in x$ if and only if there is $q \in G_\omega$ and $\langle t_n \mid n < \omega\rangle$ Prikry generic for $\dom q$ as above such that for all large $n$,   
\[q ^\smallfrown \langle t_i \mid i < n\rangle \Vdash \check{\gamma} \in j_{n,\omega}(\dot{y}_n).\]
\end{proof}

Let $k_\omega\colon N_\omega \to M_\omega$ be the natural elementary embedding.  Note
that $k_\omega$ naturally extends to an embedding between the two  Prikry generic extensions,
\[k^*_\omega\colon N_\omega[\l i_m(\kappa) \mid m < \omega\r] \longrightarrow
M_\omega[\l j_m(\kappa) \mid m<\omega \r].\]

Similarly, For each $n<\omega$ we define $W_{n,\omega}$ to be the limit ultrapower obtained by starting in $M_n$ and iterating the measure $j_n(E_\kappa)$, the normal measure of $j_n(E)$. As with $N_\omega$. Denote the critical points of the iteration by $j_n(E_\kappa)$ and its images by $\l \kappa^n_m \r_{m<\omega}$, and let $i^n_\omega \colon  V \to W_{n,\omega}$ denote the composition of the finite ultrapower map $j \colon  V \to M_n$ with the latter infinite iteration map by the normal measures. 
Therefore elements $y_n \in W_{n,\omega}$ are of the form 
$i^n_\omega(f)(\alpha_0,\dots,\alpha_{n-1},\kappa^n_0,\dots,\kappa^n_{m-1})$ 
for some $n,m < \omega$, $f \colon  \kappa^{n+m} \to V$ in $V$, 
and $\alpha_{\ell} \in j_{\ell}(\kappa^{++})$ for each $\ell < n$.

The structures $W_{n,\omega}$ are naturally connected via maps
$k_{n,r} \colon W_{n,\omega} \to W_{r,\omega}$ for $n \leq r < \omega$, given by 
\[ \begin{matrix}
k_{n,r}(i^n_\omega(f)(\alpha_0,\dots,\alpha_{n-1},\kappa^n_0,\dots,\kappa^n_{m-1}) ) \\ 
= \,\, i^r_\omega(f)(\alpha_0,\dots,\alpha_{n-1},\kappa^r_0,\dots,\kappa^r_{m-1})
\end{matrix}
\]

It is straightforward to verify that the limit of the directed system 
\[\{ W_{n,\omega}, k_{n,r} \mid n \leq r < \omega\}\] is 
$M_\omega$, and the direct limit maps
$k_{n,\omega} \colon  W_{n,\omega} \to M_\omega$, which are defined by 

\[
\begin{matrix}
k_{n,r}\left( i^n_\omega(f)(\alpha_0,\dots,\alpha_{n-1},\kappa^n_0,\dots,\kappa^n_{m-1})  \right) & \\
   \, \ = j_\omega(f)(\alpha_0,\dots,\alpha_{n-1}, j_{n}(\kappa), j_{n+1}(\kappa),\dots, j_{n+m-1}(\kappa))&  
\end{matrix}
\]

${}$\\
naturally extend to the generic extensions by the suitable Prikry sequences. Indeed, denote $W_{n,\omega}[\l j_m(\kappa) \mid m<n \r \fr \l i_m(j_n(\kappa)) \mid m < \omega\r]$ by $W^*_{n,\omega}$, for each $n  <\omega$.
Then $k_{n,\omega}$ lifts to: 
\[k^*_{n,\omega}\colon W_{n,\omega}^*\longrightarrow M_\omega[\l j_m(\kappa) \mid m<\omega\r ].\]

Finally, we note that following implies that 
$M_\omega[\l j_m(\kappa) \mid m < \omega \r]$ is the direct limit of the system of Prikry generic extensions  
\[\l W_{n,\omega}^*, k^*_{n,k} \mid n \leq k < \omega\r.\]

\begin{claim} \label{komegaH} ${<}k^*_\omega``H{>} \in M_\omega[G_\omega][\mathcal{H}]$
is 
generic for $j_\omega(\dot{\qo})^{\l j_n(\kappa) \mid n<\omega \r}$ over
$M_\omega[\l j_n(\kappa)\mid n<\omega \r]$. 
\end{claim}

\begin{proof} 

Let $D \in M_\omega[\l j_n(\kappa) \mid n < \omega \r]$ be a dense open subset
of the forcing $j_\omega(\dot{\qo})^{\l j_n(\kappa) \mid n <\omega \r}$.  Then there are
$n<\omega$ and $\bar{D} \in W_{n,\omega}^*$ such that $k^*_{n,\omega}(\bar{D}) = D$.
It follows from our arguments above that $<k_n``H>$ is generic for
$k_n(\mathbb{Q}_\omega)$ over $M_n$.  Now $k_{n}(\qo_\omega) \in W_{n,\omega}^*$ and
$\bar{D}$ is a dense subset of it.  So $<k_{n}``H> \cap \bar{D}$ is nonempty.  Let
$q \in H$ be such that $k_{n}(q) \in \bar{D}$.
It follows that $k^*_\omega(q) = k^*_{n,\omega}(k^*_n(q)) \in D$.
So we have shown that $<k^*_\omega``H>$ is generic over $M_\omega[\l j_n(\kappa)
\mid n <\omega \r]$.  It remains to show that it is a member of
$M_\omega[G_\omega][\mathcal{H}]$.  Since $<k^*_\omega``H> =
<k^*_{n,\omega}``<k^*_n``H>>$, we have that $<k^*_\omega``H> \in M_n[\mathcal{H}_n]$
for all $n<\omega$.  By Claim \ref{intersection2}, it follows that
$<k^*_\omega``H> \in M_\omega[G_\omega][\mathcal{H}]$. \end{proof}

To complete the proof of Theorem \ref{mainthm3} we need a finer control of the
relationship between $\mathbb{P}$ and $\mathbb{P}^*$ names.  To this end we make
some definitions.

\begin{definition} Let $f \in \mathbb{P}^*$, $\alpha$ be an ordinal and
$\dot{C}$ be a $\mathbb{P}^*$-name for a club subset of $\kappa^+$.  
We say that $f$
\emph{stably forces} $\check{\alpha} \in \dot{C}$ ($f \Vdash^s_{\mathbb{P}^*} \check{\alpha}
\in \dot{C}$) if every alteration of $f$ on fewer than $\kappa$ many coordinates
forces $\check{\alpha} \in \dot{C}$. \end{definition}

We have the following straightforward claims.

\begin{claim} If $g$ extends $f$ in $\po^*$ and $f \Vdash^s \check{\alpha} \in
\dot{C}$, then $g \Vdash^s \check{\alpha} \in \dot{C}$. \end{claim}

\begin{claim} If $f \Vdash^s \check{\alpha} \in \dot{C}$, then for every finite
sequence $\vec{\nu}$ from some tree associated to $f$, $f \frown \vec{\nu}
\Vdash^s \check{\alpha} \in \dot{C}$.  If in addition $f = f' \frown \vec{\nu}$
for some finite sequence $\vec{\nu}$, then $f' \Vdash^s \check{\alpha} \in
\dot{C}$. \end{claim}

We define $\dot{C}^s = \{ (\check{\alpha},f) \mid f \Vdash^s \check{\alpha} \in
\dot{C} \}$.  Clearly it is forced that that $\dot{C}^s \subseteq \dot{C}$.  It
is also straightforward to see that $\dot{C}^s$ is forced to be closed.

\begin{claim}\label{claim: unbounded stable club} $\dot{C}^s$ is forced to be unbounded in $\kappa^+$. \end{claim}

\begin{proof} Fix some $f \in \mathbb{P}^*$ and $\alpha_0< \kappa^+$.  Take some
sufficiently large $\theta$ and some $N \prec H_\theta$ of size $\kappa$ such
that $\mathbb{P}^*,f,\dot{C},\alpha_0 \in N$ and ${}^{<\kappa}N \subseteq N$.
Since $\po^*$ is $\kappa^+$-closed we can find a $(\po^*,N)$-generic condition
$f^* \geq f$ with $\dom(f^*) = N \cap \kappa^{++}$.  Let $\alpha = \sup(N \cap
\kappa^{++})$.

Let $f'$ be any condition obtained by altering $f^*$ on a set of size less than
$\kappa$.  Since ${}^{<\kappa}N \subseteq N$, the alteration is a member of $N$.
Hence a standard argument shows that $f'$ is also $(\po^*,N)$-generic and so $f'
\Vdash \check{\alpha} \in \dot{C}$. \end{proof}

The name $\dot{C}^s$ behaves well when translated to a $\mathbb{P}$-name.  Let
\[\dot{E} = \{ (\check{\alpha},p) \mid f^p \Vdash^s \check{\alpha} \in \dot{C}
\}.\]

\begin{claim} $\dot{E}$ is forced by $\mathbb{P}$ to be a club subset
of $\kappa^+$. \end{claim}

\begin{proof} 
By the arguments of Claim \ref{claim: unbounded stable club}, $\dot{E}$ is unbounded.
Let us show first that $\dot{E}$ is forces to be closed. 

Let $p = \l f, T\r$ be a condition that forces that $\delta$ is an accumulation point 
of $\dot{E}$.
Therefore, for every $\alpha < \delta$ there are
$\vec{\nu} = \l \nu_0,\dots,\nu_{n-1}\r \in T$, 
$p' = \l f' , T'\r \leq^* p \fr \vec{\nu}$, and $\gamma < \delta$ such that 
$f' \Vdash^s \gamma \in \name{C}$. Clearly, $f'$ is a Cohen extension of $f \fr \vec{\nu}$. 
Denote by $f^*$ the Cohen extension of $f$ for which $f' = f^* \fr \vec{\nu}$. 
Since $f' \Vdash^s \gamma \in \name{C}$, $f^* \Vdash^s \gamma \in \name{C}$ as well. It follows that the condition $p^* = \l f^*,T^*\r$, $T^* = \pi_{\dom(f^*),\dom(f)}^{-1}(T)$ is a direct extension of $p$ and forces ''$\can{\gamma} \in \dot{E}$''.

Using the fact the Cohen extension order on the function $f$ is $\kappa^+$-closed and $\cf(\delta) \leq \kappa$, we can repeat this process and construct a sequence of Cohen extensions $\l f_i \mid i \leq \cf(\delta)\r$ of $f$, and an increasing sequence $\l \gamma_i \mid i < \cf(\delta)\r$ cofinal in $\delta$, such that 
$f_i \Vdash^s \can{\gamma_i} \in \name{C}$. 
Let $\bar{f} = f_{\cf(\delta)}$.
Since $\name{C}$ is a Cohen name of a club, $\bar{f} \Vdash^s \can{\delta} \in \dot{C}$. Setting $\bar{T} = \pi_{\dom(\bar{f}),\dom(f)}^{-1}(T)$ and 
 $\bar{p} = \l \bar{f},\bar{T}\r$, we have that $\bar{p}$ is a direct extension of $p$, and forces $\can{\delta} \in \dot{E}$. 
\end{proof}

\begin{remark} \label{closed-extension} The proofs of the previous two claims
work for any $\mathbb{P}^*$-name for a club $\dot{C}$ in any $\kappa^+$-closed generic
extension. \end{remark}

We are now ready to finish the proof of Theorem \ref{mainthm3}.

\begin{claim} In $M_\omega[G_\omega][\mathcal{H}]$, every finite collection of
stationary subsets of $j_\omega(\kappa^+)$ reflects at a common point.
\end{claim}

\begin{proof} Let $S_i$ for $i < m$ be a sequence of stationary sets in
$M_\omega[G_\omega][\mathcal{H}]$.  By Claim \ref{komegaH}, ${<}k_\omega``H{>}$ is a
club in $j_\omega(\kappa^+)$ disjoint from
$S_{j_\omega(\kappa)}^{j_\omega(\kappa^+)}$ as computed in $M_\omega$ and
${<}k_\omega``H{>} \in M_\omega[G_\omega][\mathcal{H}]$.  So we can assume that each
$S_i$ concentrates on a fixed cofinality below $j_\omega(\kappa)$.  For
simplicity we assume that the cofinalities of the $S_i$ are bounded below
$\kappa$.

Let $T_n^i$ be the set of $\alpha$ such that $j_{n,\omega}(\alpha) \in
S_i$.  For a fixed $n<\omega$, $T_n^i$ for $i < m$ is definable in
$M_n[G_n][\mathcal{H}_n]$.  By the indestructibility of stationary reflection at
$j_n(\kappa^+)$ in $M_n$, if each $T_n^i$ for $i < m$ is stationary, then they
reflect at a common point.  Suppose that $\delta$ is this common reflection
point.  Then there are sets $A_i\in M_n$ for $i <k$ with $\otp(A)=\cf(\delta)$
such that $A_i$ is stationary in $\delta$ and $A_i \subseteq T_n^i$.  It follows
that for each $i<m$, $j_{n,\omega}(A_i) = j_{n,\omega}``A_i \subseteq S_i$ and
hence the collection of $S_i$ reflect at $j(\delta)$.

We assume for the sake of a contradiction that for each $n<\omega$ at least one
of the $T_n^i$ is nonstationary.  Let $\dot{C}_n$ be a name whose interpretation
by $G_n$ and $\mathcal{H}_n$, $C_n$, is a club disjoint from some $T_n^{i_n}$.
By the arguments above and Remark \ref{closed-extension}, we can work in
$M_n[\mathcal{H}_n]$ and translate $\dot{C}_n$ to a $j_n(\mathbb{P})$-name
$\dot{E}_n$ for a club subset of $j_n(\kappa^+)$.  By the construction of
$\dot{E}_n$, we have that for $\alpha < j_n(\kappa^+)$, $j_{n,\omega}(\alpha)$
is in $j_{n,\omega}(\dot{E}_n)^{G_\omega}$ if and only if it is in
$j_{n,\omega}(\dot{C}_n)^{{<}j_{n,\omega}``G_n{>}}$.

By the closure of $M_\omega[G_\omega]$ under
$\omega$-sequences, $\langle j_{n,\omega}(\dot{E}_n) \mid n<\omega \rangle$ is
in $M_\omega[G_\omega]$.  Hence we can interpret it using $G_{\omega}$ and
$j_{n,\omega}``\mathcal{H}_n$.  Let $D_n$ denote the resulting club.  Let $i^*<m$ be such that $i^* = i_n$ for infinitely many $n$.  We claim
that $\bigcap_{n<\omega}D_n$ is disjoint from $S_{i^*}$.  Otherwise, we have
$j_{n,\omega}(\alpha) \in D_n \cap S_{i^*}$ for some $n$ such that $i^* = i_n$.
It follows that $\alpha \in C_n \cap T_n^{i_n}$, a contradiction. \end{proof}
\subsection{Down to $\aleph_{\omega}$}
In the paper \cite{SigmaPrikryIII} it is shown that it is consistent relative to infinitely many supercompact cardinals, that stationary reflection at $\aleph_{\omega+1}$ together with the failure of $\SCH$ at $\aleph_{\omega}$. 

Let us show how to obtain the same result in our case, starting with the assumption of $\kappa^+$-$\Pi^1_1$-subcompact that carries a $(\kappa,\kappa^{++})$-extender, as in the previous part. 
\begin{theorem}
Let $\kappa$ be a $\kappa^+$-$\Pi^1_1$-subcompact cardinal and let us assume that it carries a $(\kappa,\kappa^{++})$-extender. Then, there is a forcing extension in which $\kappa=\aleph_\omega$, $2^{\aleph_{\omega}} = \aleph_{\omega+2}$ and $\GCH$ holds below $\aleph_{\omega}$ and every finite collection of stationary subsets of $\aleph_{\omega+1}$ reflect simultaneously.
\end{theorem}

The proof of this theorem is essentially the same as the proof of Theorem \ref{mainthm3}. We follow the argument as in the previous subsection and explain the differences. 

Let us prepare the universe so that $\GCH$ holds and the subcompactness of $\kappa$ is indestructible under Cohen forcing $\Add(\kappa^+,\kappa^{++})$. Let $j_E \colon V \to M_1$ be the extender ultrapower embedding. Let us assume moreover that the approachability ideal is always maximal. Namely, for every regular uncountable cardinal $\rho$, $\rho \in I[\rho]$

We define the iterations $\langle j_{n,m} \colon M_n\to M_m \mid n < m \leq \omega\rangle$, and $\langle \iota_{n,m} \colon N_n \to N_m \mid n < m \leq \omega\rangle$ as before.

By standard arguments (see for example \cite[Lemma 8.5 and Proposition 15.1]{CummingsHandbook}, and use Claim \ref{claim:bounding-the-width-of-k} for the computation of the width), there is a $M_1$-generic filter $K\subseteq \Col(\kappa^{+3}, <j_E(\kappa))$. Moreover, by elementarity for every $n$, the filter $j_n(K)$ is $M_{n+1}$-generic. Fix $K_0 \subseteq \Col(\omega_4, <\kappa)$ a $V$-generic and let 
\[\mathcal{K} = \langle K_0\rangle^\smallfrown \langle j_n(K) \mid n < \omega\rangle\]

Let us define the filters $G_n, \mathcal{H}_n$ and the models $M_n$ for $n \leq \omega$ exactly as before and let us consider the model $M_{\omega}[G_\omega][\mathcal{H}][\mathcal{K}]$. Let $K_n = \mathcal{K} \restriction n + 1$.

In \cite{HayutBD}, it is shown that the parallel of Lemma \ref{intersection1}, namely $M_\omega[G_\omega][\mathcal{K}] = \bigcap M_n[G_n][K_n]$, holds.

The next point in the proof in which we need to modify the argument slightly is Lemma \ref{intersection2}, in which we need to show \[M_\omega[G_\omega][\mathcal{K}][\mathcal{H}] = \bigcap M_n[G_n][K_n][\mathcal{H}_n].\] 

As there are elementary embeddings \[j_{n, \omega} \colon M_n[\mathcal{H}_n] \to M_\omega[<j_{n,\omega}``\mathcal{H}_n>] \subseteq M_\omega[\mathcal{H}]\] the argument of \cite{HayutBD} for extender based Prikry forcing with interleaved collapses works without further modifications. 

The claims for the existence of a generic for the forcing that shoots a club disjoint from $j_\omega(S^{\kappa^{+}}_{\kappa})$, Claim \ref{komegaH}, remains the same. We are now ready for the proof of the main theorem of this subsection.
\begin{theorem}\label{thm:down-to-aleph-omega}
In $M_{\omega}[G_\omega][\mathcal{H}][\mathcal{K}]$ $\SCH$ fails at $j_{\omega}(\kappa)$ which is $\aleph_{\omega}$ and every finite collection of stationary subsets of $j_{\omega}(\kappa^+)$ reflects simultaneously.  
\end{theorem}
\begin{proof}
Let us verify first that the only cardinals which are collapsed are the cardinals in the intervals $[\omega_4, \kappa) \cup [\kappa^{+3}, j_1(\kappa)) \cup \cdots [j_n(\kappa^{+3}), j_{n+1}(\kappa)) \cup \cdots$.

As mentioned above,
\[M_{\omega}[G_\omega][\mathcal{H}][\mathcal{K}] 
= \bigcap M_{n}[G_n][\mathcal{H}_n][K_n]\] 
and in particular, every $M_\omega$-cardinal below $j_\omega(\kappa)$ which is collapsed in the extension must be collapsed in $M_{n}[G_n][\mathcal{H}_n][K_n]$ for all large $n$. For cardinals above $j_{\omega}(\kappa)$ one can use the continuity of the elementary embeddings in order to show that they do not change cofinality. 
Therefore, the arguments for the failure of $\SCH$ remains the same. 

Let us prove the stationary reflection. For simplicity, let us show that every stationary subset of $j_\omega(\kappa^+)$ reflects in this model. The general case is similar.

There are two components in the proof which are slightly different: First, given a stationary set $S \in M_{\omega}[G_\omega][\mathcal{H}][\mathcal{K}]$, one can pull it to a set $T_n \in M_n[G_n][\mathcal{H}_n][K_n]$ exactly as before and check whether it is stationary there. Note that the forcing adding $K_n$ is of cardinality $j_n(\kappa)$. So, if $T_n$ is stationary in $M_n[G_n][\mathcal{H}_n][K_n]$ then there is a single condition $q \in K_n$ and stationary set $T'_n \in M_n[G_n][\mathcal{H}_n]$ such that $q \Vdash T'_n \subseteq T_n$. 

Then, by the $j_n(\kappa^+)$-subcompactness of $j_n(\kappa)$ in $M_n[G_n][\mathcal{H}_n]$ one can conclude that $T_n'$ has a reflection point $\delta$ of cofinality $<j_n(\kappa)$ and above $j_{n-1}(\kappa)^{+++}$ (see \cite{HU} for details). This is where the assumption about the approachability plays a role: we conclude that this stationary subset of $\delta$ remains stationary in $M_n[G_n][\mathcal{H}_n][\mathcal K_n]$. So, there is a stationary subset of $\delta$ of order type $\cf \delta$ in $M_n$, $T_n{''}\subseteq T_n \cap \delta$. Since $j_{n,\omega}(T_n'') = j_{n,\omega} \text{ `` } T_n''$, we conclude that $M_{\omega}[G_\omega][\mathcal{H}][K_n]\models j_{n,\omega}(T_n'') \subseteq S$.

By downwards absoluteness of stationarity and continuity of the embedding $j_{n,\omega}$, $j_{n,\omega}(T_n'')$ is stationary at $j_{n,\omega}(\delta)$ in $M_{\omega}[G_\omega][\mathcal{H}][K_n]$ and thus $S$ reflects at $j_{n,\omega}(\delta)$ in this model. 

To summarize, if there is $n$ such that $T_n$ is stationary, then $S$ reflects in $M_\omega[G_\omega][\mathcal{H}][\mathcal{K}]$. 

Otherwise, for all $n$, $T_n$ is non-stationary. In this case, for each $n$ there is a club name for a club $\dot{C}_n$ disjoint from $T_n$ and we would like to push them (or subclubs of them) to $M_{\omega}[G_\omega][\mathcal{H}][\mathcal{K}]$. 

For every $n$, the forcing introducing $K_n$ is $j_n(\kappa)$-cc, and therefore, for every club $C_n \in M_n[G_n][\mathcal{H}_n][K_n]$ there is a subclub $C'_n \in M_n[G_n][\mathcal{H}_n]$. We are now in exactly the same situation of the proof of Theorem \ref{mainthm3}: we obtain for each $n$ a stable name $\dot{E}_n$ which is forced to be a subclub of $C'_n$ and thus there is in $M_{\omega}[G_\omega][\mathcal{H}]$ a subclub of the set $\bigcap j_{n,\omega}(\dot{E}_n)$ which is forced to be disjoint from $S$. 
\end{proof}
\begin{remark}
The same method can be used to obtain arbitrary finite gap in the failure of $\SCH$. 
\end{remark}

\begin{thebibliography}{10}

\bibitem{Bukovski}
Lev Bukovsk\'{y}, \emph{Iterated ultrapower and {P}rikry's forcing}, Comment.
  Math. Univ. Carolinae \textbf{18} (1977), no.~1, 77--85. \MR{0446978}

\bibitem{CummingsHandbook}
James Cummings, \emph{Iterated forcing and elementary embeddings}, Handbook of
  set theory. {V}ols. 1, 2, 3, Springer, Dordrecht, 2010, pp.~775--883.
  \MR{2768691}

\bibitem{Dehornoy}
Patrick Dehornoy, \emph{Iterated ultrapowers and {P}rikry forcing}, Ann. Math.
  Logic \textbf{15} (1978), no.~2, 109--160. \MR{514228}

\bibitem{FT}
Matthew Foreman and Stevo Todorcevic, \emph{A new {L}\"{o}wenheim-{S}kolem
  theorem}, Trans. Amer. Math. Soc. \textbf{357} (2005), no.~5, 1693--1715.
  \MR{2115072}

\bibitem{Gitik-new}
Moti Gitik, \emph{Blowing up the power of a singular cardinal of uncountable
  cofinality}, preprint.

\bibitem{Woodin-surgery}
\bysame, \emph{The negation of the singular cardinal hypothesis from
  {$o(\kappa)=\kappa^{++}$}}, Ann. Pure Appl. Logic \textbf{43} (1989), no.~3,
  209--234. \MR{1007865}

\bibitem{Gitik-HB}
\bysame, \emph{Prikry-type forcings}, Handbook of set theory. {V}ols. 1, 2, 3,
  Springer, Dordrecht, 2010, pp.~1351--1447. \MR{2768695}

\bibitem{gitikap}
\bysame, \emph{Another method for constructing models of not approachability
  and not {SCH}},  (2018), preprint.

\bibitem{gitiktree}
\bysame, \emph{Another model with the tree property and not {SCH}},  (2018),
  preprint.

\bibitem{GitikReflNotSCH}
\bysame, \emph{Reflection and not {SCH} with overlapping extenders}, Arch.
  Math. Logic \textbf{61} (2022), no.~5-6, 591--597. \MR{4452118}

\bibitem{gitiksharon}
Moti Gitik and Assaf Sharon, \emph{On {SCH} and the approachability property},
  Proc. Amer. Math. Soc. \textbf{136} (2008), no.~1, 311--320. \MR{2350418}

\bibitem{HayutBD}
Yair Hayut, \emph{Prikry type forcings and the bukovsk\'y-dehornoy phenomena},
  2023.

\bibitem{HU}
Yair Hayut and Spencer Unger, \emph{Stationary reflection},  (2018), submitted.

\bibitem{magidor}
Menachem Magidor, \emph{Reflecting stationary sets}, J. Symbolic Logic
  \textbf{47} (1982), no.~4, 755--771 (1983). \MR{683153}

\bibitem{Merimovich}
Carmi Merimovich, \emph{Prikry on extenders, revisited}, Israel J. Math.
  \textbf{160} (2007), 253--280. \MR{2342498}

\bibitem{Merimovich2}
\bysame, \emph{Extender-based {M}agidor-{R}adin forcing}, Israel J. Math.
  \textbf{182} (2011), 439--480. \MR{2783980}

\bibitem{MerimovichGenCriteria}
\bysame, \emph{Mathias like criterion for the extender based {P}rikry forcing},
  Ann. Pure Appl. Logic \textbf{172} (2021), no.~9, Paper No. 102994, 6.
  \MR{4264147}

\bibitem{neeman}
Itay Neeman, \emph{Aronszajn trees and failure of the singular cardinal
  hypothesis}, J. Math. Log. \textbf{9} (2009), no.~1, 139--157. \MR{2665784}

\bibitem{SigmaPrikryI}
Alejandro Poveda, Assaf Rinot, and Dima Sinapova, \emph{Sigma-{P}rikry forcing
  {I}: the axioms}, Canad. J. Math. \textbf{73} (2021), no.~5, 1205--1238.
  \MR{4325864}

\bibitem{SigmaPrikryIII}
Alejandro Poveda, Assaf Rinot, and Dima Sinapova, \emph{Sigma-prikry forcing
  iii: Down to $\aleph_\omega$}, 2022.

\bibitem{sharon}
Assaf Sharon, \emph{Weak squares, scales, stationary reflection and the failure
  of sch}, Ph.D. Thesis, Tel Aviv University (2005).

\bibitem{shelah}
Saharon Shelah, \emph{Reflection implies the {SCH}}, Fund. Math. \textbf{198}
  (2008), no.~2, 95--111. \MR{2369124}

\bibitem{sinapova3}
Dima Sinapova, \emph{A model for a very good scale and a bad scale}, J.
  Symbolic Logic \textbf{73} (2008), no.~4, 1361--1372. \MR{2467223}

\bibitem{sinapova1}
\bysame, \emph{The tree property and the failure of {SCH} at uncountable
  cofinality}, Arch. Math. Logic \textbf{51} (2012), no.~5-6, 553--562.
  \MR{2945567}

\bibitem{sinapova2}
\bysame, \emph{The tree property and the failure of the singular cardinal
  hypothesis at {$\aleph_{\omega^2}$}}, J. Symbolic Logic \textbf{77} (2012),
  no.~3, 934--946. \MR{2987144}

\bibitem{sinapovaunger}
Dima Sinapova and Spencer Unger, \emph{Scales at {$\aleph_\omega$}}, Israel J.
  Math. \textbf{209} (2015), no.~1, 463--486. \MR{3430248}

\end{thebibliography}
\providecommand{\bysame}{\leavevmode\hbox to3em{\hrulefill}\thinspace}
\providecommand{\MR}{\relax\ifhmode\unskip\space\fi MR }
\providecommand{\MRhref}[2]{%
  \href{http://www.ams.org/mathscinet-getitem?mr=#1}{#2}
}
\providecommand{\href}[2]{#2}

\end{document}